\documentclass[11pt]{amsart}
\usepackage[margin=2.8cm]{geometry}
\usepackage{amssymb,colordvi}
\usepackage{multirow}
\usepackage{bm}
\usepackage{accents}
\usepackage{float}
\usepackage{mathrsfs}
\usepackage{hyperref}
\usepackage{graphicx,color}
\numberwithin{equation}{section}
\theoremstyle{example}
\newtheorem{example}{Example}[section]
\newtheorem{theorem}{Theorem}
\numberwithin{theorem}{section}

\newtheorem{proposition}[theorem]{Proposition}
\newtheorem{convention}[theorem]{Convention}
\newtheorem{lemma}[theorem]{Lemma}

\newtheorem{definition}[theorem]{Definition}

\newtheorem{notation}[theorem]{Notation}

\newcommand{\Z}{\mathbb{Z}}

\newcommand{\oc}[1]{\accentset{\circ}{#1}}

\DeclareMathOperator{\sign}{sign}

\DeclareMathOperator{\val}{val}
\DeclareMathOperator{\Val}{Val}
\DeclareMathOperator{\trop}{trop}

\DeclareMathOperator{\Conv}{Conv}
\DeclareMathOperator{\Coef}{Coef}
\DeclareMathOperator{\Baar}{Bar}
\DeclareMathOperator{\ord}{ord}
\DeclareMathOperator{\coef}{coef}

\newcommand{\R}{\mathbb{R}}

\DeclareMathOperator{\Int}{Int}

\DeclareMathOperator{\MV}{MV}
\usepackage{latexsym,stmaryrd}

\title[Constructing polynomial systems with many positive solutions]{Constructing polynomial systems with many positive solutions using tropical geometry}

\author{Boulos El Hilany}
\address{ Eberhard Karls Universit\"at T\"ubingen\\
          Fachbereich Mathematik\\
          Auf der Morgenstelle 10\\
          72076 T\"ubingen, Deutschland}

\email{boulos.el-hilany@uni-tuebingen.de}

\begin{document}
\maketitle
 
\begin{abstract}
The number of positive solutions of a system of two polynomials in two variables defined in the field of real numbers with a total of five distinct monomials cannot exceed 15. All previously known examples have at most 5 positive solutions. Tropical geometry is a powerful tool to construct polynomial systems with many positive solutions. The classical combinatorial patchworking method arises when the tropical hypersurfaces intersect transversally. In this paper, we prove that a system as above constructed using this method has at most $6$ positive solutions. We also show that this bound is sharp. Moreover, using non-transversal intersections of tropical curves, we construct a system as above having $7$ positive solutions.
\end{abstract}

\section{Introduction and statement of the main results} The \textit{support} of a system of (Laurent) polynomials is the set of points $w\in\mathbb{Z}^n$ corresponding to monomials $x^w=x_1^{w^1}\cdots x_n^{w^n}$ appearing in that system with non-zero coefficient. Consider a system\begin{equation}\label{eq:sys:Intro}
f_1(x_1,\ldots,x_n)=\cdots = f_n(x_1,\ldots,x_n)=0,
\end{equation} of polynomials defined in $\R[x_1^{\pm 1},\ldots,x_n^{\pm 1}]$ and supported on a set $\mathcal{W}\subset\mathbb{Z}^n$. Such real polynomial systems appear frequently in pure and applied mathematics (c.f.~\cite{BR90,GH02,By89,DRRS07}), and in many cases we are interested in studying their real solutions. It is a classical problem in algebraic geometry to count such solutions, and this turns out to be a difficult task  especially when dealing with polynomials of high degree or high number of monomials. We often restrict the problem to find an upper bound on the number of real solutions to a given system~\eqref{eq:sys:Intro}. One could apply B\'ezout's Theorem using the degrees of the polynomials, or Bernstein-Kouchnirenko's results~\cite{Ber75,Kus75} using the volumes of their Newton polytopes $\Delta(f_i)$. However, since these classical methods also hold true for solutions in the torus $(\mathbb{C}^*)^n$, one rarely obtains a precise estimation. A natural question then arises is whether there exists an upper bound on the number of real solutions to a given system~\eqref{eq:sys:Intro} that depends only on the number of points in its support $\mathcal{W}$. 

Assume that we have $|\mathcal{W}|=n+k+1$ for some positive integer $k$, and that all the solutions of~\eqref{eq:sys:Intro} in $(\mathbb{C}^*)^n$ are non-degenerate (i.e. the Jacobian matrix evaluated at each such solution has full rank). This implies that such a system has a finite number of solutions. An important breakthrough due to Khovanskii~\cite{Kh91} was proving that the maximal number of non-degenerate \textit{positive solutions} (i.e. contained in the positive orthant of $\R^n$) of~\eqref{eq:sys:Intro} is bounded above by \begin{equation*}\label{eq:KovBound}
2^{n+k \choose 2}(n+1)^{n+k}.
\end{equation*} The positive solutions  of~\eqref{eq:sys:Intro} are indeed of great interest since giving an upper bound $N_{|\mathcal{W}|}$ on their number that depends on the values $n,k\geq 1$, one deduces the upper bound $2^nN_{|\mathcal{W}|}$ on the number of real non-degenerate solutions to~\eqref{eq:sys:Intro}. Khovanskii's bound is far from being sharp since it comes as a consequence of an even bigger result involving solutions in $(\mathbb{R}_{>0})^n$ of polynomial functions in logarithms of the coordinates and monomials. Nevertheless, it is the first upper bound that is independent of the degrees and the Newton polytopes for systems~\eqref{eq:sys:Intro} and an arbitrary number $n$. 

In~\cite{BS07}, F. Bihan and F. Sottile significantly reduced Khovanskii's bound by showing that there are fewer than \begin{equation}\label{Intro:eq:bound:BS}
\frac{e^2 + 3}{4}2^{k \choose 2}n^k
\end{equation} non-degenerate positive solutions to~\eqref{eq:sys:Intro}. This new bound is asymptotically sharp in the sense that for a fixed $k$ and big enough $n$, there exist systems~\eqref{eq:sys:Intro} having $O(n^k)$ positive solutions. However the bound~\eqref{Intro:eq:bound:BS} is not sharp for systems with special structure (e.g. with prescribed number of monomials in each equation). On the other hand, sharp upper bounds on the number of positive solutions are already known in some special cases. For example, Descartes' rule of sign states that the univariate polynomial obtained from~\eqref{eq:sys:Intro} when supposing $n=1$ has the maximum of $k+1$ positive solutions (counted with multiplicities). Also, F. Bihan proved in~\cite{B07} that if $k=1$, then $n+1$ is a sharp upper bound on the number of positive solutions to~\eqref{eq:sys:Intro}. 

One of the first cases where the sharp upper bound on the number of non-degenerate positive solutions is not known is the case of a bivariate polynomial system of two equations having five distinct points in its support. It was also proven in~\cite{BS07} that a sharp bound to such a system (of type $n=k=2$ for short) is not greater than $15$. On the other hand, the best constructions had only $5$ non-degenerate positive solutions. The first such published example, made by B. Haas~\cite{H02}, is a construction consisting of two real bivariate trinomials. Other examples of such systems having 5 positive solutions were later constructed in~\cite{DRRS07}. The authors in the latter paper also showed that such systems are rare in the following sense. They study the discriminant variety of coefficients spaces of polynomial systems composed of two bivariate trinomials with fixed exponent vectors, and show that the chambers (connected components of the complement) containing systems with the maximal number of positive solutions are small.

In this paper, we consider real systems of type $n=k=2$ in their full generality (i.e. not only the case of two trinomials). The motivation behind this paper is to adopt some of the tools developed in \emph{tropical geometry} in order to construct real polynomial systems of type $n=k=2$ that give more than five positive solutions. Tropical geometry is a new domain in mathematics that is situated at the junction of fields such as toric geometry, complex or real geometry, and combinatorics~\cite{Mik06,MR05,MS15}. It turns out that Sturmfels' Theorem~\cite{S94} can be reformulated in the context of tropical geometry (see~\cite{Mik04,Rul01}). This makes the latter an effective tool to construct polynomial systems with prescribed support and many positive solutions. The principal idea is to consider a family of polynomial systems \begin{equation}\label{Eq:sys:param}
P_{1,t}(x,y)=P_{2,t}(x,y)=0,
\end{equation} of type $n=k=2$ with special 1-parametrized coefficients $a_i^j(t)$ for $(i,j)\in\{1,2\}\times\{1,\ldots,5\}$. We then associate to $P_{1,t}$ and $P_{2,t}$ \emph{tropical curves} $T_1,T_2\subset\R^2$ (see Subsection~\ref{Prel:Subs:TropPH}). These are piecewise-linear combinatorial objects that keep track of much of the information about the (parametrized) solutions of~\eqref{Eq:sys:param}. Assume first that the associated tropical curves intersect \emph{transversally} in a finite set of points $\mathcal{S}$ (i.e. the cardinality of $\mathcal{S}$ does not change after perturbations). Then, Sturmfels' generalization of Viro's Theorem (see Theorem~\ref{Prel:Th:Stur}) states that there exists a bijection between the positive solutions to the real system obtained from~\eqref{Eq:sys:param} by taking $t$ small enough, and a subset $\mathcal{S}_+\subseteq\mathcal{S}$ of \emph{positive tropical transversal points} (c.f. Definition~\ref{Prel:Def:PosPar}). Therefore, similarly to the famous Viro's combinatorial patchworking (c.f. Theorem~\ref{Prel:Th:Vir}), the construction of real polynomial systems with many positive solutions becomes a combinatorial problem. If the system~\eqref{Eq:sys:param} is of type $n=k=2$, then the number of transversal intersection points of $T_1$ and $T_2$ is bounded from above by $6$. It was previously unknown whether this upper bound can be attained. We prove that this bound is sharp and can be realized by positive transversal intersection points. 

\begin{proposition}\label{Prop:Intro}
There exist two plane tropical curves $T_1$ and $T_2$ defined by equations containing a total of five monomials and which have six positive transversal intersection points.
\end{proposition} Due to Theorem~\ref{Prel:Th:Stur}, the construction made for proving the latter result also gives a construction of a real polynomial system of type $n=k=2$ that has six positive solutions. Furthermore it is clear from Theorem~\ref{Prel:Th:BihDisc} that one cannot hope to improve the result in Proposition~\ref{Prop:Intro} when restricting to polynomial systems of type $n=k=2$ with tropical curves intersecting transversally.

Consequently, in order to obtain a better construction, we consider real parametrized polynomial systems~\eqref{Eq:sys:param} of type $n=k=2$ whose tropical curves $T_1$ and $T_2$ intersect in a non-empty set that does \emph{not} consist of transversal points. A consequence of an important result due to Kapranov~\cite{Kap00} is that the set $T_1\cap T_2$ contains the \emph{tropicalizations} of the solutions of~\eqref{Eq:sys:param}. For each linear piece $\xi$ of a connected component of $T_1\cap T_2$, we associate a real \emph{reduced} polynomial system extracted from~\eqref{Eq:sys:param} (see Definition~\ref{Def:RedSys}) and prove that it encodes all \emph{positive} solutions $(\alpha_1(t),\alpha_2(t))$ of~\eqref{Eq:sys:param} which tropicalize in $\xi$ (by positive, we mean that the first-order terms of $\alpha_1(t)$ and $\alpha_2(t)$ have positive coefficients). If $\xi$ is of dimension zero, then results in~\cite{Kat09,Rab12,OP13} and~\cite{Br-deMe11} show that $\xi$ lifts to solutions to~\eqref{Eq:sys:param}, and then such non-degenerate solutions $(\alpha_1(t),\alpha_2(t))$ which are positive can be estimated by computing the real reduced system of~\eqref{Eq:sys:param} with respect to $\xi$ (see Proposition~\ref{Prop:fromR:toK}). If $\xi$ has dimension 1, then a method was developed in~\cite{E16} to compute the positive solutions that tropicalize in the relative interior of $\xi$. The latter methods for non-transversal linear pieces of dimension $0$ and $1$ are sufficient to construct a real polynomial system of type $n=k=2$ having more than six positive solutions. Namely, we prove our main result.

\begin{theorem}\label{Th:Main}
There exists a real polynomial system of type $n=k=2$ that has seven solutions in $(\R_{>0})^2$.
\end{theorem} The strategy behind the construction of a system satisfying Theorem~\ref{Th:Main} goes as follows. First, we show that to any system~\eqref{Eq:sys:param} of type $n=k=2$, one can associate a \emph{normalized system}, which is easier to deal with, that has the same number of non-degenerate positive solutions as~\eqref{Eq:sys:param}. A case-by-case analysis was made in~\cite{E16} to identify the few classes of candidates of normalized systems that have more than six positive solutions. The construction described in the present paper is based on one such candidate.

This paper is organized as follows. We introduce in Section~\ref{Sec:Tropical} some basic notions of tropical geometry. In Section~\ref{Sec:Transversal}, we give a description of the tropical reformulation of Viro's Patchworking Theorem and its generalization followed by the proof of Proposition~\ref{Prop:Intro}. Finally, Section~\ref{Sec:Non-Transversal} is devoted to the proof of Theorem~\ref{Th:Main} .

\textbf{Acknowledgements:} I am very grateful to Fr\'ed\'eric Bihan for fruitful discussions and guidance. I also would like to thank Pierre-Jean Spaenlehauer for computations that approximated the real positive solutions to the real system that was constructed to prove Theorem~\ref{Th:Main}.

\section{A brief introduction to tropical geometry}\label{Sec:Tropical} We state in this section some of the well-known facts about tropical geometry, much of the exposition and notations in this section are taken from~\cite{Br-deMe11,BB13,Ren15}. For more information about the topic, the reader may refer to~\cite{MS15,IMS09} for example.

\begin{definition}
A \textbf{polyhedral subdivision} of a convex polytope $\Delta\subset\R^n$ is a set of convex polytopes $\{\Delta_i\}_{i\in I}$ such that

\begin{itemize}

 \item $\cup_{i\in I}\Delta_i=\Delta$, and\\
 
 \item if $i,j~\in I$, then if the intersection $\Delta_i\cap\Delta_j$ is non-empty, it is a common face of the polytope $\Delta_i$ and the polytope $\Delta_j$.
\end{itemize}
\end{definition}

\begin{definition}\label{Prel:Def:Regular}
Let $\Delta$ be a convex polytope in $\R^n$ and let $\tau$ denote a polyhedral subdivision of $\Delta$ consisting of convex polytopes. We say that $\tau$ is \textbf{regular} if there exists a continuous, convex, piecewise-linear function $\varphi:~\Delta\rightarrow \R$ which is affine linear on every simplex of $\tau$.
\end{definition} Let $\Delta$ be an integer convex polytope in $\R^n$ and let $\phi:~\Delta~\cap~\mathbb{Z}^n\rightarrow \R$ be a function. We denote by $\hat{\Delta}(\phi)$ the convex hull of the graph of $\phi$, i.e., $$\hat{\Delta}(\phi):=\Conv\left(\{(i,\phi(i))~\in\R^{n+1}~|~i\in\Delta\cap\mathbb{Z}^n\}\right).$$ Then the polyhedral subdivision of $\Delta$, induced by projecting the union of the lower faces of $\hat{\Delta}(\phi)$ onto the first $n$ coordinates, is regular. We will shortly describe $\phi$ using the polynomials that we will be working with.

\subsection{Tropical polynomials and hypersurfaces}\label{Prel:Subs:TropPH}
A \textbf{locally convergent generalized Puiseux series} is a formal series of the form $$a(t)=\underset{r\in R}{\sum} \alpha_rt^r ,$$ where $R\subset\mathbb{R}$ is a well-ordered set, all $\alpha_r\in\mathbb{C}$, and the series is convergent for $t>0$ small enough. 
We denote by $\mathbb{K}$ the set of all locally convergent generalized Puiseux series. It is naturally a field of characteristic 0, which turns out to be algebraically closed. 

\begin{notation}
Let $\coef(a(t))$ denote the coefficient of the first term of $a(t)$ following the increasing order of the exponents of $t$. We extend $\coef$ to a map $\Coef:~\mathbb{K}^n\rightarrow\mathbb{R}^n$ by taking $\coef$ coordinate-wise, i.e. $\Coef(a_1(t),\ldots , a_n(t))=(\coef(a_1(t)),\ldots , \coef(a_n(t)))$
\end{notation}

 An element $a(t)=\underset{r\in R}{\sum} \alpha_rt^r$ of $\mathbb{K}$ is said to be \textbf{real} if $\alpha_r\in\mathbb{R}$ for all $r$, and \textbf{positive} if $a(t)$ is real and $\coef(a(t))>0$. Denote by $\mathbb{RK}$ (resp. $\mathbb{RK}_{>0}$) the subfield of $\mathbb{K}$ composed of real (resp. positive) series. Since elements of $\mathbb{K}$ are convergent for $t>0$ small enough, an algebraic variety over $\mathbb{K}$ (resp. $\mathbb{RK}$) can be seen as a one-parametric family of algebraic varieties over $\mathbb{C}$ (resp. $\mathbb{R}$). The field $\mathbb{K}$ has a natural non-archimedian valuation defined as follows:

$$ 
  \begin{array}{lccc}
    \displaystyle \val: & \mathbb{K} & \longrightarrow & \mathbb{R}\cup\{-\infty\} \\
    
    \displaystyle \ & 0 & \longmapsto  & -\infty \\
    
     \displaystyle \ & \underset{r\in R}{\sum} \alpha_rt^r\neq 0 & \longmapsto  & -\min_R\{r\ |\ \alpha_r\neq 0\}. \\
  \end{array}
$$ The valuation extends naturally to a map $\Val:~\mathbb{K}^n\rightarrow(\mathbb{R}\cup\{-\infty\})^n$ by evaluating $\val$ coordinate-wise, i.e. $\Val(z_1,\ldots , z_n)=(\val(z_1),\ldots , \val(z_n))$. We shall often use the notation $\val$ and $\Val$ when the context is a \textit{tropical polynomial} or a \textit{tropical hypersurface}. On the other hand, define $\ord:=-\val$, with $\ord(0)=+\infty$, and use it as a notation when the context is an element in $\mathbb{RK}^n$ or a polynomial in $\mathbb{RK}[z_1^{\pm 1},\ldots,z_2^{\pm 1}]$.

\begin{convention}\label{Conv:zeroOrd}
For any $s\in\mathbb{K}$, we have $\coef(s)=0\Leftrightarrow s=0$ and $\ord(s)=+\infty\Leftrightarrow s=0$
\end{convention}

Consider a polynomial $$f(z):=\sum_{w\in\mathcal{W}}c_wz^{w}\in\mathbb{K}[z_1^{\pm 1},\ldots,z_n^{\pm 1}],$$ with $\mathcal{W}$ a finite subset of $\mathbb{Z}^n$ and all $c_w$ are non-zero. Let $V_f=\{z\in(\mathbb{K^*})^2~|~f(z)=0\}$ be the zero set of $f$ in $(\mathbb{K}^*)^n$ 


The \textbf{tropical hypersurface} $V^{\trop}_f$ associated to $f$ is the closure (in the usual topology) of the image under $\Val$ of $V_f$: $$V^{\trop}_f=\overline{\Val(V_f)}\subset\R^n,$$ endowed with a \textit{weight function} which we will define later. There are other equivalent definitions of a tropical hypersurface. Namely, define 

$$ 
  \begin{array}{lccc}
    \displaystyle \nu: & \mathcal{W} & \longrightarrow & \mathbb{R} \\
    
    \displaystyle \ & w & \longmapsto  & \ord(c_w). \\
    
  \end{array}
$$ Its \textbf{Legendre transform} is a piecewise-linear convex function $$ 
  \begin{array}{lccc}
    \displaystyle \mathcal{L}(\nu): & \R^n & \longrightarrow & \mathbb{R} \\
    
    \displaystyle \ & x & \longmapsto  &\displaystyle \max_{w\in\mathcal{W}}\{\langle x, w\rangle - \nu(w)\}, \\
    
  \end{array}
$$ where $\langle ~,~\rangle:~\R^n\times\R^n\rightarrow\R$ is the standard eucledian product. The set of points $x\in\R^n$ at which $\mathcal{L}(\nu)$ is not differentiable is called the \textbf{corner locus} of $\mathcal{L}(\nu)$.
 We have the fundamental Theorem of Kapranov~\cite{Kap00}.

\begin{theorem}[Kapranov]\label{Th:Kap}
A tropical hypersurface $V_f^{\trop}$ is the corner locus of $\mathcal{L}(\nu)$.
\end{theorem} Tropical hypersurfaces can also be described as algebraic varieties over the \textit{tropical semifield} $(\R\cup \{-\infty\}, ``+", ``\times")$, where for any two elements $x$ and $y$ in $\R\cup \{-\infty\}$, one has $$``x+y"=\max(x,y)\quad \text{and}\quad ``x\times y"=x+y.$$ A multivariate tropical polynomial is a polynomial in $\R[x_1,\ldots,x_n]$, where the addition and multiplication are the tropical ones. Hence, a tropical polynomial is given by a maximum of finitely many affine functions whose linear parts have integer coefficients and constant parts are real numbers. The tropicalization of the polynomial $f$ is a tropical polynomial $$f_{\trop}(x)=\max_{w\in\mathcal{W}}\{\langle x, w\rangle +\val(c_w)\}.$$ This tropical polynomial coincides with the piecewise-linear convex function $\mathcal{L}(\nu)$ defined above. Therefore, Theorem \ref{Th:Kap} asserts that $V_f^{\trop}$ is the corner locus of $f_{\trop}$. Conversely, the corner locus of any tropical polynomial is a tropical hypersurface.

\begin{example} 
\normalfont A polynomial $f\in\mathbb{RK}[z_1,z_2]$ with equation 

\begin{equation}\label{eq:Ex:Deg3}
f(z_1,z_2)=-t +z_1 -tz_1^2-z_1z_2+z_2 +tz_2^2,
\end{equation} its associated tropical polynomial is $$f_{trop}(x_1,x_2)=\max\{-1,x_1,2x_1-1,x_1+x_2,x_2,2x_2-1\},$$ and the corresponding tropical hypersurface is shown in Figure~\ref{Fig:ExampleCon} on the left.

\begin{figure}[H]
\centering
\includegraphics[scale=0.7]{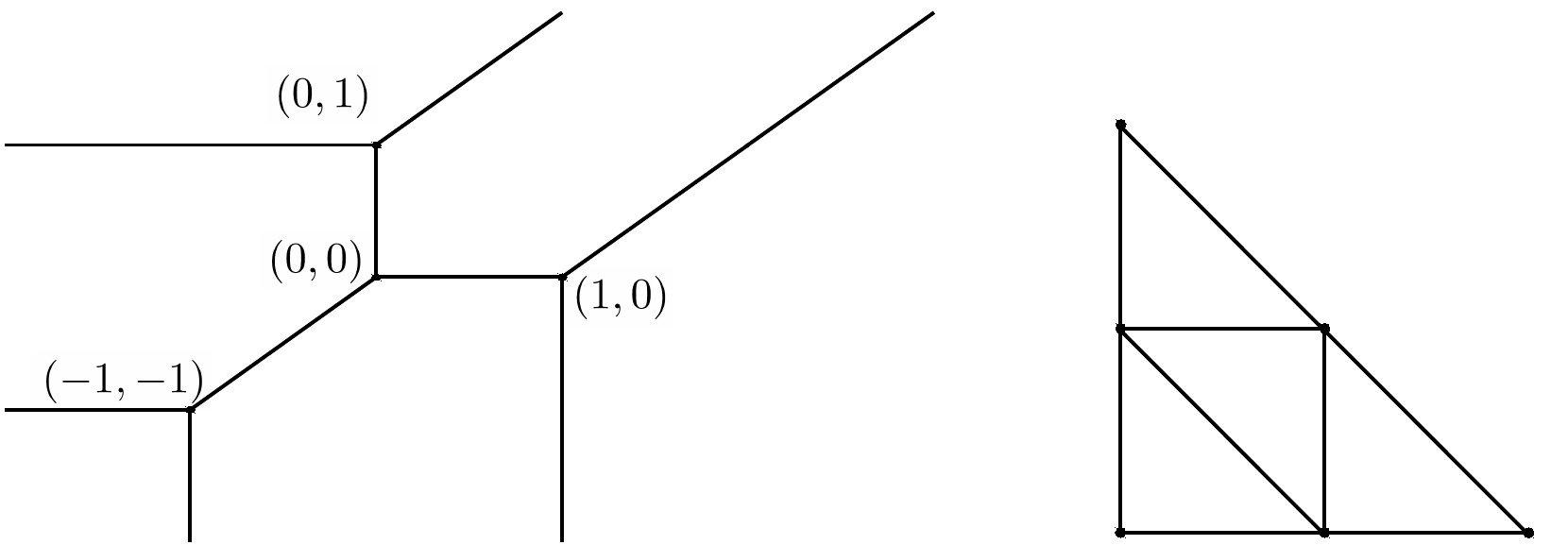}
\caption{An example of a tropical conic in $\mathbb{R}^2$, and its dual subdivision.}
\label{Fig:ExampleCon}
\end{figure}

\end{example}

\subsection{Tropical hypersurfaces and subdivisions}\label{Prel:Subsec:TropHS}

A tropical hypersurface induces a subdivision of the Newton polytope $\Delta(f)$ in the following way (see right side of Figure~\ref{Fig:ExampleCon}). The hypersurface $V_f^{\trop}$ is a $(n-1)$-dimensional piecewise-linear complex which induces a polyhedral subdivision $\Xi$ of $\R^n$. We will call \textbf{cells} the elements of $\Xi$. Note that these cells have rational slopes. The $n$-dimensional cells of $\Xi$ are the closures of the connected components of the complement of $V^{\trop}_f$ in $\R^n$. The lower dimensional cells of $\Xi$ are contained in $V^{\trop}_f$ and we will just say that they are cells of $V_f^{\trop}$. 

Consider a cell $\xi$ of $V_f^{\trop}$ and pick a point $x$ in the relative interior of $\xi$. Then the set $$\mathcal{I}_x=\{w\in\mathcal{W}~|~\exists~ x\in\mathbb{R}^n,~f_{trop}(x)=\langle x,w\rangle + \val(c_w)\}$$ is independent of $x$, and denote by $\Delta_{\xi}$ the convex hull of this set. All together the polyhedra $\Delta_{\xi}$ form a subdivision $\tau$ of $\Delta(f)$ called the \textbf{dual subdivision}, and the cell $\Delta_{\xi}$ is called the \textbf{dual} of $\xi$. Both subdivisions $\tau$ and $\Xi$ are dual in the following sense. There is a one-to-one correspondence between $\Xi$ and $\tau$, which reverses the inclusion relations, and such that if $\Delta_\xi\in\tau$ corresponds to $\xi\in\Xi$ then\\

\begin{enumerate}

 \item $\dim\xi + \dim\Delta_\xi=n$,\\
 
 \item the cell $\xi$ and the polytope $\Delta_\xi$ span orthogonal real affine spaces, and\\
 
 \item the cell $\xi$ is unbounded if and only if $\Delta_\xi$ lies on a proper face of $\Delta(f)$.\\

\end{enumerate}

Note that $\tau$ coincides with the regular subdivision of Definition~\ref{Prel:Def:Regular}. Indeed, let $\hat{\Delta}(f)\subset\R^n\times\R$ be the convex hull of the points $(w,\nu(w))$ with $w\in\mathcal{W}$ and $\nu(w)=\ord(c_w)$. Define $$ 
  \begin{array}{lccc}
    \displaystyle \hat{\nu}: & \Delta(f) & \longrightarrow & \mathbb{R} \\
    
    \displaystyle \ & x & \longmapsto  & \min\{y~|~(x,y)\in\hat{\Delta}(f)\}.\\
    
  \end{array}
$$ Then, the the domains of linearity of $\hat{\nu}$ form the dual subdivision $\tau$.

 Consider a facet (face of dimension $n-1$) $\xi$ of $V_f^{\trop}$, then $\dim\Delta_\xi=1$ and we define the \textbf{weight} of $\xi$ by $w(\xi):=Card(\Delta_\xi\cap\mathbb{Z}^n)-1$. Tropical varieties satisfy the so-called balancing condition. Since in this paper, we only work with tropical curves in $\R^2$, we give here this property only for this case. We refer to~\cite{Mik06} for the general case. Let $T\subset\R^n$ be a tropical curve, and let $v$ be a vertex of $T$. Let $\xi_1,\ldots,\xi_l$ be the edges of $T$ adjacent to $v$. Since $T$ is a rational graph, each edge $\xi_i$ has a primitive integer direction. If in addition we ask that the orientation of $\xi_i$ defined by this vector points away from $v$, then this primitive integer vector is unique. Let us denote by $u_{v,i}$ this vector.

\begin{proposition}[Balancing condition]\label{Prop:Balance}
For any vertex $v$, one has $$\sum_{i=1}^lw(\xi_i)u_{v,i}=0.$$
\end{proposition}

\subsection{Intersection of tropical hypersurfaces}\label{Prel:Subsec:InterComp} Consider polynomials $f_1,\ldots,f_r~\in\mathbb{K}[z_1^{\pm 1},\ldots,z_n^{\pm 1}]$. For $i=1,\ldots,r$, let $\Delta_i\subset\R^n$ (resp. $T_i\subset\R^n$) denote the Newton polytope (resp. tropical curve) associated to $f_i$. Recall that each tropical curve $T_i$ defines a piecewise linear polyhedral subdivision $\Xi_i$ of $\R^n$ which is dual to a convex polyhedral subdivision $\tau_i$ of $\Delta_i$. The union of these tropical hypersurfaces defines a piecewise-linear polyhedral subdivision $\Xi$ of $\R^n$. Any non-empty cell of $\Xi$ can be written as $$\xi=\xi_1\cap\cdots\cap\xi_r$$ with $\xi_i\in\Xi_i$ for $i=1,\ldots,r$. We require that $\xi$ does not lie in the boundary of any $\xi_i$, thus any cell $\xi\in\Xi$ can be uniquely written in this way. Denote by $\tau$ the mixed subdivision of the Minkowski sum $\Delta=\Delta_1 +\cdots + \Delta_r$ induced by the tropical polynomials $f_1,\ldots,f_r$. Recall that any polytope $\sigma\in\tau$ comes with a privileged representation $\sigma=\sigma_1+\cdots +\sigma_r$ with $\sigma_i\in\tau_i$ for $i=1,\ldots,r$. The above duality-correspondence applied to the (tropical) product of the tropical polynomials gives rise to the following well-known fact (see~\cite{BB13} for instance).

\begin{proposition}\label{Prop:MixedSubd}
There is a one-to-one duality correspondence between $\Xi$ and $\tau$, which reverses the inclusion relations, and such that if $\sigma\in\tau$ corresponds to $\xi\in\Xi$, then\\

\begin{enumerate}

 \item if $\xi=\xi_1\cap\cdots\cap~\xi_r$ with $\xi_i\in\Xi_i$ for $i=1,\ldots,r$, then $\sigma$ has representation $\sigma = \sigma_1 +\cdots +\sigma_r$ where each $\sigma_i$ is the polytope dual to $\xi_i$.\\
 
 \item $\dim\xi + \dim\sigma = n$,\\
 
 \item the cell $\xi$ and the polytope $\sigma$ span orthonogonal real affine spaces,\\
 
 \item the cell $\xi$ is unbounded if and only if $\sigma$ lies on a proper face of $\Delta$.\\

\end{enumerate}
\end{proposition} 


\begin{definition} A cell $\xi$ is \textbf{transversal} if it satisfies $\dim(\Delta_\xi)=\dim(\Delta_{\xi_1})+\cdots+\dim(\Delta_{\xi_r})$, and it is \textbf{non transversal} if the previous equality does not hold. 
\end{definition}

\section{First construction: Transversal case}\label{Sec:Transversal}  Since this paper concerns algebraic sets of dimension zero contained in $(\R_{>0})^n$, the exposition in this section will only restrict to that orthant of $\R^n$. 
 
\subsection{Generalized Viro theorem and tropical reformulation}\label{Prel:Sec:ViroGen}

Following the description of B. Sturmfels~\cite{S94}, we recall now Viro's Theorem for hypersurfaces. Let $\mathcal{W}\subset\mathbb{Z}^n$ be a finite set of lattice points, and denote by $\Delta$ the convex hull of $\mathcal{W}$. Assume that $\dim\Delta = n$ and let $\varphi:~\mathcal{W}\rightarrow\Z$ be any function inducing a regular triangulation $\tau_{\varphi}$ of the integer convex polytope $\Delta$ (see Definition~\ref{Prel:Def:Regular}). Fix non-zero real numbers $c_w,~w\in\mathcal{W}$. For each positive real number $t$, we consider a Laurent polynomial 

\begin{equation}\label{Prel:eq:ViroPoly}
f_t(z_1,\ldots,z_n)=\sum_{w\in\mathcal{W}}c_wt^{\varphi(w)}z^w.
\end{equation} 

Let $\Baar(\tau_{\varphi})$ denote the first barycentric subdivision of the regular triangulation $\tau_{\varphi}$. Each maximal cell $\mu$ of $\Baar(\tau_{\varphi})$ is incident to a unique point $w\in\mathcal{W}$. We define the sign of a maximal cell $\mu$ to be the sign of the associated real number $c_w$. The sign of any lower dimensional cell $\lambda\in\Baar(\tau_{\varphi})$ is defined as follows:\\

\multirow{5}{*}{\ }{$\displaystyle \sign(\lambda):=\begin{cases} +\quad \text{if}~~\sign(\mu)=+\quad\text{for all maximal cells}\quad\mu\quad\text{containing}~\lambda,\\ -\quad \text{if}~~\sign(\mu)=-\quad\text{for all maximal cells}\quad\mu\quad\text{containing}~\lambda,\\ 0\quad\text{otherwise}.\end{cases}$}\\

 Let $\mathcal{Z}_+(\tau_\varphi,f)$ denote the subcomplex of $\Baar(\tau_\varphi)$ consisting of all cells $\lambda$ with $\sign(\lambda)=0$, and let $V_+(f_t)$ denote the zero set of $f_t$ in the positive orthant of $\R^n$. 
 Denote by $\Int(\Delta)$ the relative interior of $\Delta$.

\begin{theorem}[Viro]\label{Prel:Th:Vir}
For sufficiently small $t>0$, there exists a homeomorphism $(\R_{>0})^n\rightarrow\Int(\Delta)$ sending the real algebraic set $V_+(f_t)\subset(\R_{>0})^n$ to the simplicial complex $\mathcal{Z}_+(\tau_\varphi,f)\subset\Int(\Delta)$.
\end{theorem}

Naturally, a signed version of Theorem~\ref{Prel:Th:Vir} holds in each of the $2^n$ orthants $$(\R_{>0})^\epsilon:=\{(x_1,\ldots,x_n)\in(\R^*)^n~|~\sign(x_i)=\epsilon_i~\text{for}~i=1,\ldots,n\},$$ where $\epsilon\in\{+,-\}^n$. In fact, O. Viro proves a more general version of Theorem~\ref{Prel:Th:Vir}, in which he defines a set that is homeomorphic to the zero set $V(f_t)\subset\R^n$ (not only the positive zero set $V_+(f_t)$) by means of gluing the zero sets of $f_t$ contained in all other orthants of $\R^n$.

We now reformulate Theorem~\ref{Prel:Th:Vir} using tropical geometry. We consider $g:=f_t$ as a polynomial defined over the field of real generalized locally convergent Puiseux series, where each coefficient $c_wt^{\varphi(w)}\in\mathbb{RK}^*$ of $g$ has only one term. Therefore $\coef(c_wt^{\varphi(w)})=c_w$, $\val(c_wt^{\varphi(w)})=-\varphi(w)$, and we associate to $g$ a tropical hypersurface $V_g^{\trop}$ as defined in Subsection~\ref{Prel:Subs:TropPH}. Recall that $V_g^{\trop}$ induces a subdivision $\Xi_g$ of $\mathbb{R}^n$ that is dual to $\tau_\varphi$. The tropical hypersurface $V_g^{\trop}$ is homeomorphic to the barycentric subdivision $\Baar(\tau_\varphi)$. Indeed, $\tau_\varphi$ is a triangulation, and thus $\Baar(\tau_\varphi)$ becomes dual to $\tau_\varphi$ in the sense of the duality described in Subsection~\ref{Prel:Subsec:TropHS}. 

We define for each $n$-cell $\xi\in\Xi_g$, dual to a $0$-face (vertex) $w$ of the triangulation $\tau_{\varphi}$, a sign $\epsilon(w)\in\{+,-\}$, to be equal to the sign of $c_w$. 

\begin{definition}\label{Prel:Def:PosPar}
The \textbf{positive part}, denoted by $V_{g,+}^{\trop}$, is the subcomplex of $V_g^{\trop}$ consisting of all $(n-1)$-cells of $V_g^{\trop}$ that are adjacent to two $n$-cells of $V_g^{\trop}$ having different signs (see the left part of Figure~\ref{Fig:PosPart} for an example). A \textbf{positive facet} $\xi_+$ is an $(n-1)$-dimensional cell of $V_{g,+}^{\trop}$.
\end{definition} The following is a Corollary of Mikhalkin~\cite{Mik04} and Rullgard~\cite{Rul01} results, where they completely describe the topology of $V(f_t)$ using \textit{amoebas}.

\begin{figure}[H]
\centering
\includegraphics[scale=0.6]{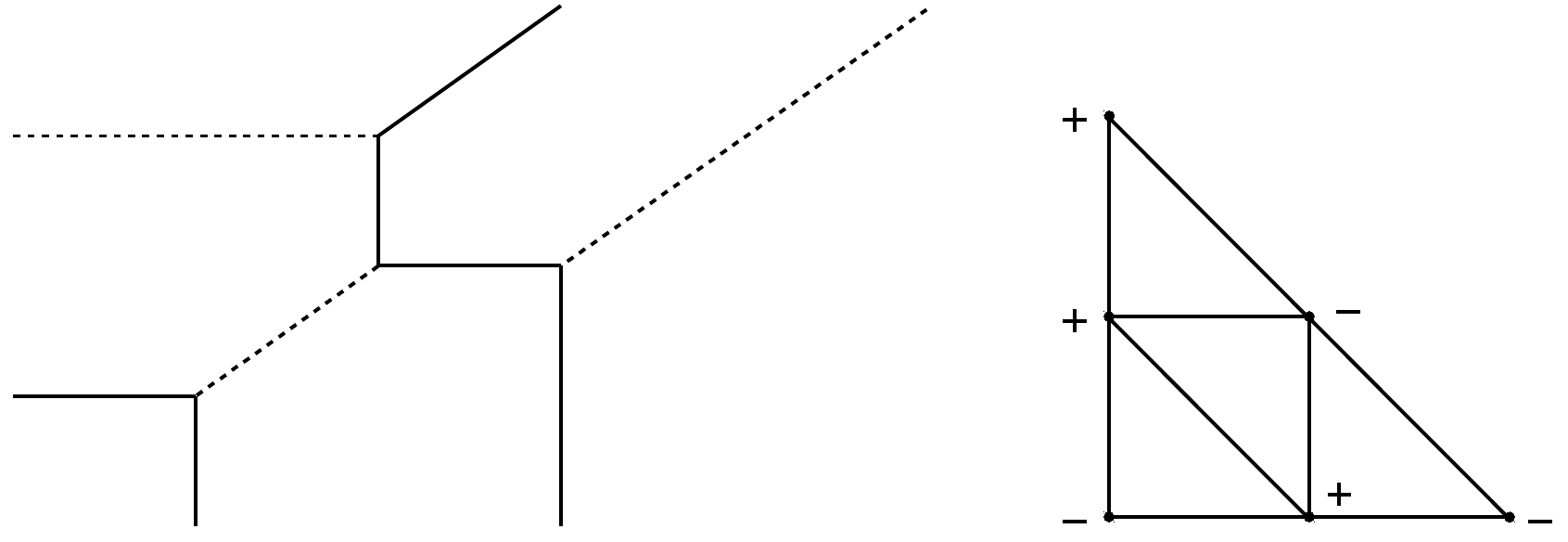}
\caption{The positive part of the tropical hypersurface associated to $-t +z_1 -tz_1^2-z_1z_2+z_2 +tz_2^2$ is represented as the union of the solid segments and solid half-rays.}
\label{Fig:PosPart}
\end{figure}

\begin{theorem}[Mikhalkin, Rullgard]\label{Prel:Th:Trop-Reform}
For sufficiently small $t>0$, there exists a homeomorphism $(\R_{>0})^n\rightarrow\R^n$ sending the zero set $V_+(f_t)\subset(\R_{>0})^n$ to $V_{g,+}^{\trop}\subset\mathbb{R}^n$.
\end{theorem}

B. Sturmfels generalized Viro's method for complete intersections in~\cite{S94}. We give now a tropical reformulation of one of the main Theorems of~\cite{S94}. 

Consider a system 

\begin{equation}\label{Prel:eq:sys:Sturm}
f_{1,t}(z_1,\ldots,z_n)=\cdots=f_{r,t}(z_1,\ldots,z_n)=0,
\end{equation} of $r$ equations, where all $f_{t,i}$ are polynomials of the form~\eqref{Prel:eq:ViroPoly}. For $i=1,\ldots,r$, we define as before $g_i:=f_{i,t}$ as a polynomial in $\mathbb{RK}[z_1^{\pm 1},\ldots,z_n^{\pm 1}]$. Let $V_+(f_{1,t},\ldots,f_{r,t})\subset(\R_{>0})^n$ denote the locus of positive solutions of~\eqref{Prel:eq:sys:Sturm}. 

\begin{theorem}[Sturmfels]\label{Prel:Th:Stur}
Assume that the tropical hypersurfaces $V_{g_1}^{\trop},\ldots,V_{g_r}^{\trop}$ intersect transversally. Then for sufficiently small $t>0$, there exists a homeomorphism $(\R_{>0})^n\rightarrow\R^n$ sending the real algebraic set $\mathcal{Z}_+(f_{1,t},\ldots,f_{r,t})\subset(\R_{>0})^n$ to the intersection $V_{g_1,+}^{\trop}\cap\cdots\cap V_{g_r,+}^{\trop}\subset\R^n$.
\end{theorem} Similarly to O. Viro's work, B. Sturmfels generalizes Theorem~\ref{Prel:Th:Stur} for the zero set

\noindent $V(f_{1,t},\ldots,f_{r,t})\subset\R^n$ (see~\cite[Theorem 5]{S94}).

\subsection{Tropical transversal intersection points for bivariate polynomials} For the rest of this section, we assume that the system appearing in~\eqref{Prel:eq:sys:Sturm} has two equations in two variables (i.e. $n=r=2$), and that the tropical curves $V_{g_1}^{\trop}$and $V_{g_2}^{\trop}$ intersect transversally. Then the intersection set $V_{g_1,+}^{\trop}\cap V_{g_2,+}^{\trop}$ is a (possibly empty) set of points in $\R^2$. Each point of $V_{g_1,+}^{\trop}\cap V_{g_2,+}^{\trop}$ is expressed in a unique way as a transversal intersection $\xi_{1,+}\cap\xi_{2,+}$, where for $i=1,2$, the cell $\xi_{i,+}\subset V_{g_i,+}^{\trop}$ is a positive cell. In this section, we use Theorem~\ref{Prel:Th:Stur} to prove Proposition~\ref{Prop:Intro}.

F. Bihan~\cite{B14} gave an upper bound on $|V_{g_1}^{\trop}\cap V_{g_2}^{\trop}|$ (and thus on $|V_{g_1,+}^{\trop}\cap V_{g_2,+}^{\trop}|$) for a bivariate system~\eqref{Prel:eq:sys:Sturm} in two equations. Namely, given two finite sets $\mathcal{W}_1$ and $\mathcal{W}_2$ in $\R^2$, and for any non-empty $I\subset \{1,2\}$, write $\mathcal{W}_I$ for the set of points $\sum_{i\in I}w_i$ over all $w_i\in\mathcal{W}_i$ with $i\in I$. The associated \textit{discrete mixed volume} of $\mathcal{W}_1$ and $\mathcal{W}_2$ is defined as
\begin{equation}\label{Prel:eq:DiscMix}
\displaystyle D(\mathcal{W}_1,\mathcal{W}_2)=\sum_{I\subset[r]}(-1)^{r-|I|}|\mathcal{W}_I|,
\end{equation} where the sum is taken over all subsets $I$ of $\{1,2\}$ including the empty set with the convention that $|\mathcal{W}_\emptyset|=1$. Denote by $\mathcal{W}_i$ the support of $g_1$ for $i=1,2$. Recall that the tropical curves associated to $g_1,g_2$ intersect transversally.

\begin{theorem}[Bihan]\label{Prel:Th:BihDisc}
 The number $|V_{g_1}^{\trop}\cap V_{g_2}^{\trop}|$ is less or equal to the discrete mixed volume $D(\mathcal{W}_1,\mathcal{W}_2)$.
\end{theorem} 


When $|\mathcal{W}|=4$, then the bound of Theorem~\ref{Prel:Th:BihDisc} is $3$ and is sharp (see~\cite{B07}). However, we do not know if the discrete mixed volume bound is sharp for any polynomial system with $2$ equations in $2$ variables satisfying that the associated tropical curves intersect transversally.

\subsection{Restriction to the case $n=k=2$} Consider a system \begin{equation}\label{Statem:eq:NotNorm}
f_1 = f_2 = 0
\end{equation} of type $n=k=2$ (i.e.~\eqref{Statem:eq:NotNorm} has five distinct points in its total support), where $f_1,f_2\in\mathbb{RK}[z_1^{\pm 1},z_2^{\pm 1}]$. Assume that the tropical curves $T_1$ and $T_2$, associated to $f_1$ and $f_2$ respectively, intersect transversally. Let $\mathcal{W}_1,\mathcal{W}_2\subset\mathbb{Z}^2$ denote the supports of $f_1$ and $f_2$ respectively. 

\begin{lemma}\label{L:Discrete}
The discrete mixed volume $D(\mathcal{W}_1,\mathcal{W}_2)$ does not exceed six.
\end{lemma}

\begin{proof}Recall that $|\mathcal{W}_1\cup\mathcal{W}_2|=5$. We distinguish the five possible cases $|\mathcal{W}_1\cap\mathcal{W}_2|=i$ for $i=1,\ldots,5$, and prove the result for $i=3,4$ since the case $i=5$ is proven in~\cite{B14} and the other cases are similar. The discrete mixed volume of $\mathcal{W}_1$ and $\mathcal{W}_2$ is expressed as 

\begin{equation}\label{eq:DiscMixW1W2}
D(\mathcal{W}_1,\mathcal{W}_2)= |\mathcal{W}_1 + \mathcal{W}_2| - |\mathcal{W}_1| - |\mathcal{W}_2| + 1.
\end{equation}

Assume first that $|\mathcal{W}_1\cap\mathcal{W}_2|=4$. Then the cardinal of one of the two sets, say $\mathcal{W}_1$, is equal to four. Writing $\mathcal{W}_1=\{w_0,w_1,w_2,w_3\}$ and $\mathcal{W}_2=\{w_0,w_1,w_2,w_3,w_4\}$, we get $$\mathcal{W}_1+\mathcal{W}_2=\bigcup_{i=0}^3\{w_i+w_j~|~j=0,\ldots, 4,~j\geq i\},$$ and thus $|\mathcal{W}_1+\mathcal{W}_2|\leq 14$. Therefore, with $|\mathcal{W}_1|=4$ and $|\mathcal{W}_2|=5$, we deduce that $D(\mathcal{W}_1,\mathcal{W}_2)\leq 6$.

Assume now that $|\mathcal{W}_1\cap\mathcal{W}_2|=3$. We distinguish two cases\\

\begin{itemize}

  \item[\textbf{i)}] First case: $|\mathcal{W}_1|=3$ and $|\mathcal{W}_2|=5$ (the case where $|\mathcal{W}_1|=5$ and $|\mathcal{W}_2|=3$ is symmetric). Writing $\mathcal{W}_1=\{w_0,w_1,w_2\}$ and $\mathcal{W}_2=\{w_0,w_1,w_2,w_3,w_4\}$, we get $$\mathcal{W}_1+\mathcal{W}_2=\bigcup_{i=0}^2\{w_i+w_j~|~j=0,\ldots, 4,~j\geq i\},$$ and thus $|\mathcal{W}_1+\mathcal{W}_2|\leq 12$. Therefore, with $|\mathcal{W}_1|=3$ and $|\mathcal{W}_2|=5$, we deduce that $D(\mathcal{W}_1,\mathcal{W}_2)\leq 5$.\\

  \item[\textbf{ii)}] Second case: $|\mathcal{W}_1|=|\mathcal{W}_2|=4$. Writing $\mathcal{W}_1=\{w_0,w_1,w_2,w_3\}$ and $\mathcal{W}_2=\{w_1,w_2,w_3,w_4\}$, we get $$\mathcal{W}_1+\mathcal{W}_2=\bigcup_{i=0}^3\{w_i+w_j~|~j=1,\ldots, 4,~j\geq i\},$$ and thus $|\mathcal{W}_1+\mathcal{W}_2|\leq 13$. Therefore, with $|\mathcal{W}_1|=4$ and $|\mathcal{W}_2|=4$, we deduce that $D(\mathcal{W}_1,\mathcal{W}_2)\leq 6$.\\
\end{itemize}
\end{proof}

We finish this section by proving Proposition~\ref{Prop:Intro}.

\subsubsection*{Proof of Proposition~\ref{Prop:Intro}} Figure~\ref{Fig:SixTrans} shows that the tropical curves $T_1$ and $T_2$, associated to the equations of the system
\begin{equation}\label{n=k=2(3)}
 \begin{array}{rccl}
    \displaystyle \ & -1 + t^{12} + x^6 + x^3y^{6} -tx^{10}y^{12} & = & 0,
    \\[10pt]
    \displaystyle \ & -t^{12}  + t^5x^3y^6 -t^{1.5}x^7y^{11} + tx^{10}y^{12} & = & 0,\\
  \end{array} 
\end{equation} intersect at six transversal intersection points.\qed\\

 As explained before, Theorem~\ref{Prel:eq:sys:Sturm} shows that for a positive $t$ small enough, the system~\eqref{n=k=2(3)} becomes a real bivariate polynomial system of type $n=k=2$ having 6 positive solutions.

\begin{figure}[H]
\centering
\includegraphics[scale=1.5]{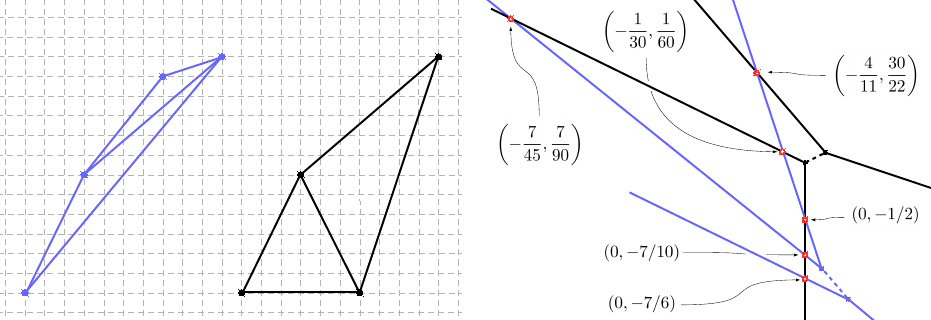}
\caption{To the left: The Newton polytopes and subdivisions associated to the equations of~\eqref{n=k=2(3)}. To the right: The tropical curves $T_1,T_2\subset\R^2$ intersecting at 6 transversal points.}
\label{Fig:SixTrans}
\end{figure}

\section{Second construction: non-transversal case}\label{Sec:Non-Transversal}

Following the notation of Subsection~\ref{Prel:Subsec:InterComp} for the case of two tropical curves in $\R^2$, we classify the types of mixed cells $\xi$ of $T_1\cap T_2$ at which the two tropical curves $T_1$ and $T_2$ intersect non-transversally. Let $\oc{\xi}$ denote the relative interior of such a linear piece $\xi$. Note that $\xi=\oc{\xi}$ if $\xi$ is a point. Consider now one linear piece $\xi:=\xi_1\cap\xi_2$ that is a result of the intersection, where $\xi_1$ and $\xi_2$ are cells of $T_1$ and $T_2$, and assume that it is non-transversal. We distinguish three types for such $\xi$.

\begin{itemize}

 \item[-] A cell $\xi$ is of \textbf{type (I)} if $\dim\xi=\dim\xi_1=\dim\xi_2=1$.\\
 
 \item[-] A cell $\xi$ is of \textbf{type (II)} if one of the cells $\xi_1$, or $\xi_2$ is a vertex, and the other cell is an edge.\\
 
 \item[-] A cell $\xi$ is of \textbf{type (III)} if $\xi_1$ and $\xi_2$ are vertices of the corresponding tropical curves.
\end{itemize}

\begin{figure}[H]
\centering
\includegraphics[scale=1.5]{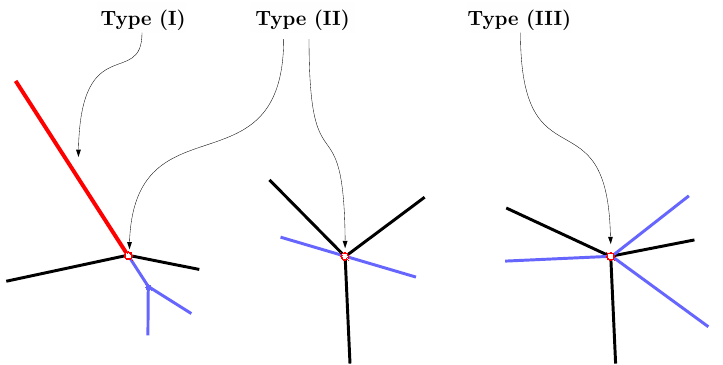}
\caption{The three types of non-transversal intersection cells.}
\label{Fig:Non-transv}
\end{figure}

\subsection{Reduced systems}

Recall that for an element $a(t)\in\mathbb{K}^*$, we denote by $\coef(a(t))$ the non-zero coefficient corresponding to the term of $a(t)$ with the smallest exponent of $t$.

\begin{definition}
Let $f=\sum_{w\in\Delta(f)\cap\mathbb{Z}^2}c_wz^w$ be a polynomial in $\mathbb{K}[z_1^{\pm 1},z_2^{\pm 1}]$ with $c_w\in\mathbb{K}^*$, and let $\xi$ denote a cell of $V_f^{\trop}$. The \textbf{reduced polynomial} $f_{|\xi}\in\mathbb{C}[z_1^{\pm 1},z_2^{\pm 1}]$ of $f$ with respect to $\xi$ is a polynomial defined as $$f_{|\xi}= \sum_{w\in\Delta_\xi\cap\mathcal{W}}\coef(c_w)z^w,$$ where $\mathcal{W}$ is the support of $f$.
\end{definition} We extend this definition to the following. Consider a system \begin{equation}\label{Prel:eq:sys1}
 f_1(z)=f_2(z)=0,
\end{equation} with $f_1,f_2$ in $\mathbb{K}[z_1^{\pm 1},z_2^{\pm 1}]$ defined as above. Assume that the intersection set $T_1\cap T_2$ of the tropical curves $T_1$ and $T_2$ is non-empty, and consider a mixed cell $\xi\in T_1\cap T_2$. As explained in Subsection~\ref{Prel:Subsec:InterComp}, the mixed cell $\xi$ is written as $\xi_1\cap\xi_2$ for some unique $\xi_1\in T_1$ and $\xi_2\in T_2$. 

\begin{definition}\label{Def:RedSys}
The \textbf{reduced system} of \eqref{Prel:eq:sys1} with respect to $\xi$ is the system $$f_{1|\xi_1}=f_{2|\xi_2}=0,$$ where $f_{i|\xi_i}$ is the reduced polynomial of $f_i$ with respect to $\xi_i$ for $i=1,2$.
\end{definition} Let $\mathcal{W}_1$ and $\mathcal{W}_2$ denote the supports of $f_1$ and $f_2$ respectively, and write $$f_1(z)=\sum_{v\in\mathcal{W}_1}a_vz^v\quad\text{and}\quad f_2(z)=\sum_{w\in\mathcal{W}_2}b_wz^w.$$ The following result generalizes to the case of a polynomial system defined on the same field with $n$ equations in $n$ variables. 

\begin{proposition}\label{Prop:from:K:toR}
If the system~\eqref{Prel:eq:sys1} has a solution $(\alpha,\beta)\in(\mathbb{K}^*)^2$ such that $\Val(\alpha,\beta)\in\oc{\xi}$, then $\left(\coef(\alpha),\coef(\beta)\right)\in(\mathbb{C}^*)^2$ is a solution of the reduced system
\begin{equation}\label{eq:red.sys0}
f_{1|\Delta_{\xi_1}}= f_{2|\Delta_{\xi_2}}=0.
\end{equation}

\end{proposition}

\begin{proof} 
Assume that~\eqref{Prel:eq:sys1} has a solution $(\alpha,\beta)\in(\mathbb{K}^*)^2$ such that $\Val(\alpha,\beta)\in\oc{\xi}$. Since $\Val(\alpha,\beta)$ belongs to the relative interior of each of $\xi_1$ and $\xi_2$, we have $$\max\{\langle \Val(\alpha,\beta), v\rangle +\val(a_v),~v\in\mathcal{W}_1\setminus(\mathcal{W}_1\cap\Delta_{\xi_1})\}<\langle\Val(\alpha,\beta), v\rangle +\val(a_v)\quad\text{for}\quad v\in\mathcal{W}_1\cap\Delta_{\xi_1}$$ and $$\max\{\langle \Val(\alpha,\beta), w\rangle +\val(b_w),~w\in\mathcal{W}_2\setminus (\mathcal{W}_2\cap\Delta_{\xi_2})\}<\langle\Val(\alpha,\beta), w\rangle +\val(b_w)\quad\text{for}\quad w\in\mathcal{W}_2\cap\Delta_{\xi_2}.$$ Consequently, since $\ord=-\val$, we have $M:=-\langle\Val(\alpha,\beta), v\rangle -\val(a_v)$ and $N:=-\langle\Val(\alpha,\beta), w\rangle -\val(b_w)$ are the orders of $f_1(\alpha,\beta)$ and $f_2(\alpha,\beta)$ respectively. Therefore, replacing $(z_1,z_2)$ by $\left( t^{\ord(\alpha)}z_1,t^{\ord(\beta)}z_2\right)$ in~\eqref{Prel:eq:sys1}, such a system becomes \begin{equation} \label{eq:sys1:xi}
  \begin{array}{ccccl}
    \displaystyle f_1\left( t^{\ord(\alpha)}z_1,t^{\ord(\beta)}z_2\right) & = & t^{M}\left(\sum_{v\in\mathcal{W}_1\cap\Delta_{\xi_1}}\coef(a_v)z^v + g_1(z)\right), \\ [10 pt]
    
    \displaystyle f_2\left( t^{\ord(\alpha)}z_1,t^{\ord(\beta)}z_2\right) & = & t^{N}\left(\sum_{w\in\mathcal{W}_2\cap\Delta_{\xi_2}}\coef(b_w)z^w + g_2(z)\right),\\
  \end{array}
\end{equation} where all the coefficients of the polynomials $g_1$ and $g_2$ of $\mathbb{RK}[z_1^{\pm 1},z_2^{\pm 1}]$ have positive orders. Since $(\alpha,\beta)$ is a non-zero solution of~\eqref{eq:red.sys0}, the system~\eqref{eq:sys1:xi} has a non-zero solution $(\alpha_0,\beta_0)$ with $\ord(\alpha_0)=\ord(\beta_0)=0$ and $\Coef(\alpha,\beta)=\Coef(\alpha_0,\beta_0)$. It follows that taking $t > 0$ small enough, we get that $\Coef(\alpha_0,\beta_0)$ is a non-zero solution of $$\sum_{v\in\mathcal{W}_1\cap\Delta_{\xi_1}}\coef(a_v)z^v=\sum_{w\in\mathcal{W}_2\cap\Delta_{\xi_2}}\coef(b_w)z^w=0.$$


\end{proof}

Note that Proposition~\ref{Prop:from:K:toR} holds true for any type of tropical intersection cell $\xi$. However, the other direction does not always hold true when $\xi$ is of type (I). Recall that a solution $(\alpha,\beta)\in(\mathbb{K}^*)^2$ is positive if $(\alpha,\beta)\in (\mathbb{RK}_{>0})^2$.

\begin{proposition}\label{Prop:fromR:toK}

Assume that $\dim\xi=0$ and that all solutions of~\eqref{Prel:eq:sys1} are non-degenerate. If the reduced system of~\eqref{Prel:eq:sys1} with respect to $\xi$ has a non-degenerate solution $(\rho_1,\rho_2)\in(\R_{>0})^2$, then~\eqref{Prel:eq:sys1} has a non-degenerate solution $(\alpha,\beta)\in(\mathbb{RK}_{>0})^2$ such that $\Val(\alpha,\beta)=\xi$ and $\Coef(\alpha,\beta)=(\rho_1,\rho_2)$.
\end{proposition}

\begin{proof} E. Brugall\'e and L. L\'opez De Medrano showed in~\cite[Proposition~3.11]{Br-deMe11} (see also~\cite{Kat09,Rab12,OP13} for more details for higher dimension and more exposition relating toric varieties and tropical intersection theory) that the number of solutions of~\eqref{Prel:eq:sys1} with valuation $\xi$ is equal to the mixed volume $\MV(\Delta_{\xi_1},\Delta_{\xi_2})$ of $\Delta_{\xi_1}$ and $\Delta_{\xi_2}$ (recall that $\Delta_{\xi}=\Delta_{\xi_1}+\Delta_{\xi_2}$). Since we assumed that~\eqref{Prel:eq:sys1} has only non-degenerate solutions in $(\mathbb{K}^*)^2$, we get $\MV(\Delta_{\xi_1},\Delta_{\xi_2})$ distinct solutions of the system~\eqref{Prel:eq:sys1} in $(\mathbb{K}^*)^2$ with given valuation $\xi$. By Proposition~\ref{Prop:from:K:toR}, if $f_1(z)=f_2(z)=0$ and $\Val(z)=\xi$, then $\Coef(z)$ is a solution of the reduced system of~\eqref{Prel:eq:sys1} with respect to $\xi$. The number of solutions in $(\mathbb{C}^*)^2$ of the reduced system is $\MV(\Delta_{\xi_1},\Delta_{\xi_2})$. Assuming that this reduced system has $\MV(\Delta_{\xi_1},\Delta_{\xi_2})$ distinct solutions in $(\mathbb{C}^*)^2$, we obtain that the map $z~\mapsto\Coef(z)$ induces a bijection from the set of solutions of~\eqref{Prel:eq:sys1} in $(\mathbb{K}^*)^2$ with valuation $\xi$ onto the set of solutions in $(\mathbb{C}^*)^2$ of the reduced system of~\eqref{Prel:eq:sys1} with respect to $\xi$.

If $z$ is a solution of~\eqref{Prel:eq:sys1} in $(\mathbb{K}^*)^2$ with $\Val(z)=\xi$ and $\Coef(z)\in(\R^*)^2$, then $z\in(\mathbb{RK}^*)^2$ since otherwise, $z,\bar{z}$ would be two distinct solutions of~\eqref{Prel:eq:sys1} in $(\mathbb{K}^*\setminus\mathbb{RK}^*)^2$ such that $\Val(z)=\Val(\bar{z})=\xi$ and $\Coef(z)=\Coef(\bar{z})$.
\end{proof}

\subsection{Normalized systems}
Recall that a polynomial system is said to be of type $n=k=2$ if it is supported on a set of five distinct points in $\mathbb{Z}^2$ and consists of two equations in two variables. In what follows, we consider a system of type $n=k=2$ defined on the field of real generalized locally convergent Puiseux series. 

\begin{definition}
A \textbf{highly non-degenerate} system is a system consisting of two polynomials in two variables, and satisfying that no three points of its support belong to a line.
\end{definition}

\begin{lemma}\label{L:Normal}

Given any highly non-degenerate system of polynomials in $\mathbb{RK}[z_1^{\pm 1},z_2^{\pm 1}]$ of type $n=k=2$, one can associate to it a highly non-degenerate system 

\begin{equation} \label{eq:s_1}
  \begin{array}{ccccl}
    \displaystyle a_0z^{w_0} + a_1z^{w_1} + a_2z^{w_2} + a_3 t^{\alpha}z^{w_3} & = & 0,\\ [4 pt]
    
    \displaystyle b_0z^{w_0} + b_1z^{w_1} + b_2z^{w_2} + b_4t^{\beta}z^{w_4} & = & 0,\\
  \end{array}
\end{equation} with equations in $\mathbb{RK}[z_1^{\pm 1},z_2^{\pm 1}]$, that has the same number of non-degenerate positive solutions, where all $a_i$ and $b_j$ are in $\mathbb{RK}^*$ and verify $\ord(a_i)=\ord(b_j)=0$, all $w_i$ are in $\mathbb{Z}^2$ and both $\alpha$, $\beta$ are real numbers.
\end{lemma}

\begin{proof}
Using linear combinations, any system of type $n=k=2$ can be reduced to a system
   
   \begin{equation} \label{eq:s_2}
   \begin{array}{lccl}
    \displaystyle  c_0t^{\alpha_0}z^{\tilde{w}_0} + c_1t^{\alpha_1}z^{\tilde{w}_1} + c_2 t^{\alpha_2}z^{\tilde{w}_2} + c_3t^{\alpha_3}z^{\tilde{w}_3} & = & 0, \\ [10 pt]
    
    \displaystyle d_0t^{\beta_0}z^{\tilde{w}_0} + d_1t^{\beta_1}z^{\tilde{w}_1} + d_2t^{\beta_2} z^{\tilde{w}_2} + d_4t^{\beta_4}z^{\tilde{w}_4} & = & 0\\
  \end{array}
   \end{equation} that has the same number of non-degenerate positive solutions, where all $c_i$ and $d_j$ are in $\in\mathbb{RK}^*$ and verify $\ord(c_i)=\ord(d_j)=0$, all $\tilde{w}_i$ are in $\mathbb{Z}^2$ and all exponents of $t$ are real numbers. Assume first that $\alpha_i-\alpha_1\neq\beta_i-\beta_1$ for $i=0,2$. By symmetry, the different possibilities of strict inequalities can be reduced to only two cases.

\begin{itemize}
 
 \item First case: $\displaystyle\alpha_0 -\alpha_1 <\beta_0 -\beta_1$ and $\displaystyle\alpha_2 -\alpha_1<\beta_2 -\beta_1$.\\ 
 Since we are interested in non-degenerate positive solutions, we may suppose that $\tilde{w}_0=(0,0)$. The system 
 \begin{equation}\label{eq:s_3}
   \begin{array}{ccl}
    \displaystyle (c_0/c_1)t^{\alpha_0-\alpha_1}z^{\tilde{w}_0} + z^{\tilde{w}_1} + (c_2/c_1) t^{\alpha_2 - \alpha_1}z^{\tilde{w}_2} + (c_3/c_1)t^{\alpha_3-\alpha_1}z^{\tilde{w}_3} & = & 0,\\[10 pt]      
    \displaystyle \tilde{c}_0t^{\alpha_0 - \alpha_1}z^{\tilde{w}_0} + \tilde{c}_2t^{\alpha_2 - \alpha_1}z^{\tilde{w}_2} + (c_3/c_1)t^{\alpha_3 -\alpha_1}z^{\tilde{w}_3} - (d_4/d_1) t^{\beta_4-\beta_1}z^{\tilde{w}_4}& = & 0\\
   \end{array}
 \end{equation} has the same number of non-degenerate positive solutions as~\eqref{eq:s_2}. Indeed, the first equation of~\eqref{eq:s_3} is obtained by dividing the first equation of~\eqref{eq:s_2} by $c_1t^{\alpha_1}$, whereas the second equation of~\eqref{eq:s_3} is obtained by dividing the first equation of~\eqref{eq:s_2} by $c_1t^{\alpha_1}$ and subtracting from it the second equation of~\eqref{eq:s_2} divided by $d_1t^{\beta_1}$. Note that $\displaystyle\coef(\tilde{c}_i) = \coef(c_i/c_1)$ and $\ord(\tilde{c}_1)=0$ for $i=0,2$. We divide both equations of~\eqref{eq:s_3} by $t^{\alpha_0 -\alpha_1}$ and set $w_3 =\tilde{w}_1 $, $w_2=\tilde{w}_3$, $w_1=\tilde{w}_2$ and $w_i=\tilde{w}_i$ for $i=0,4$. Finally replacing $(z_1,z_2)$ by $(t^kz_1,t^lz_2)$ in~\eqref{eq:s_3} for some real numbers $k$ and $l$ satisfying $\langle(k,l),w_2\rangle = \alpha_0 - \alpha_3$ and $\langle(k,l),w_1\rangle = \alpha_0 - \alpha_2$ does not change the number of non-degenerate positive solutions of~\eqref{eq:s_3}. This gives a system of the form~\eqref{eq:s_1} with the same number of non-degenerate positive solutions as~\eqref{eq:s_2}.\\

 \item Second case: $\displaystyle\alpha_0 -\alpha_1< \beta_0 -\beta_1$ and $\displaystyle\alpha_2 -\alpha_1 > \beta_2 -\beta_1$.\\
Note that this case gives $\alpha_2 -\alpha_0> \beta_2 - \beta_0$. As done before, we may suppose that $\tilde{w}_4=(0,0)$. The system 
  \begin{equation}\label{eq:s_4}
   \begin{array}{ccccl}
    \displaystyle (d_1/d_0)t^{\beta_1-\beta_0}z^{\tilde{w}_1} + (d_2/d_0)t^{\beta_2 - \beta_0}z^{\tilde{w}_2} + (d_4/d_0)t^{\beta_4-\beta_0}z^{\tilde{w}_4} + z^{\tilde{w}_0} & = & 0,\\[10 pt]   
     
    \displaystyle \tilde{d}_1t^{\beta_1 - \beta_0}z^{\tilde{w}_1} + \tilde{d}_2 t^{\beta_2 - \beta_0} z^{\tilde{w}_2}  - (c_3/c_0)t^{\alpha_3-\alpha_0}z^{\tilde{w}_3} + (d_4/d_0)t^{\beta_4 -\beta_0}z^{\tilde{w}_4} & = & 0 \\
   \end{array}
  \end{equation} has the same number of non-degenerate positive solutions as~\eqref{eq:s_2}. Indeed, the first equation of~\eqref{eq:s_4} is obtained by dividing the second equation of~\eqref{eq:s_2} by $d_0t^{\beta_0}$, whereas the second equation of~\eqref{eq:s_4} is obtained by dividing the second equation of~\eqref{eq:s_2} by $d_0t^{\beta_0}$ and subtracting from it the first equation of~\eqref{eq:s_2} divided by $c_0t^{\alpha_0}$. Note that $\displaystyle\coef(\tilde{d}_i) = \coef(d_i/d_0)$ and $\ord(\tilde{d}_i)=0$ for $i=1,2$. We divide both equations of~\eqref{eq:s_4} by $t^{\beta_4 -\beta_0}$ and set $w_0 =\tilde{w}_4 $, $w_4=\tilde{w}_0$ and $w_i=\tilde{w}_i$ for $i=1,2,3$. Finally replacing $(z_1,z_2)$ by $(t^kz_1,t^lz_2)$ in~\eqref{eq:s_4} for some real numbers $k$ and $l$ satisfying $\langle(k,l),w_1\rangle = \beta_4 - \beta_1$ and $\langle(k,l),w_2\rangle = \beta_4 - \beta_2$ does not change the number of non-degenerate positive solutions of~\eqref{eq:s_5}. This gives a system of the form~\eqref{eq:s_1} with the same number of non-degenerate positive solutions as~\eqref{eq:s_2}.
\end{itemize}

Assume now that we have $\alpha_i-\alpha_1 = \beta_i-\beta_1$ for either $i=0$ or $i=2$. The case where we have equality for both $i=0$ and $i=2$ is trivial. Without loss of generality, we may suppose that $\alpha_0-\alpha_1 = \beta_0-\beta_1$ and $\alpha_2-\alpha_1 < \beta_2-\beta_1$. Note that this case gives $\beta_0 - \beta_2< \alpha_0 -\alpha_2$. Since we are interested in non-degenerate positive solutions, we may suppose that $\tilde{w}_0=(0,0)$. The system 
\begin{equation} \label{eq:s_5}
   \begin{array}{ccl}
    \displaystyle (d_0/d_2)t^{\beta_0 - \beta_2}z^{\tilde{w}_0} + (d_1/d_2)t^{\beta_1 - \beta_2}z^{\tilde{w}_1} + z^{\tilde{w}_2} + (d_4/d_2)t^{\beta_4 - \beta_2}z^{\tilde{w}_4} & = & 0,\\[10 pt]      
    \displaystyle \tilde{d}_0t^{\beta_0 - \beta_2}z^{\tilde{w}_0} + \tilde{d}_1t^{\beta_1 - \beta_2}z^{\tilde{w}_1} - (c_3/c_2)t^{\alpha_3 -\alpha_0}z^{\tilde{w}_3} + (d_4/d_2) t^{\beta_4-\beta_2}z^{\tilde{w}_4}& = & 0\\
   \end{array}
 \end{equation} has the same number of non-degenerate positive solutions of~\eqref{eq:s_2}. Indeed, the first equation of~\eqref{eq:s_5} is obtained by dividing the second equation of~\eqref{eq:s_2} by $d_2t^{\beta_2}$, whereas the second equation of~\eqref{eq:s_5} is obtained by dividing the second equation of~\eqref{eq:s_2} by $d_2t^{\beta_2}$ and subtracting from it the first equation of~\eqref{eq:s_2} divided by $c_2t^{\alpha_2}$. Note that $\displaystyle\coef(\tilde{d}_i) = \coef(d_i/d_2)$ and $\ord(\tilde{d}_i)=0$ for $i=0,1$. We divide both equations of~\eqref{eq:s_5} by $t^{\beta_0 -\beta_2}$ and set $w_2 =\tilde{w}_4 $, $w_4=\tilde{w}_2$ and $w_i=\tilde{w}_i$ for $i=0,1,3$. Finally replacing $(z_1,z_2)$ by $(t^kz_1,t^lz_2)$ in~\eqref{eq:s_5} for some real numbers $k$ and $l$ satisfying $\langle(k,l),w_1\rangle = \beta_1 - \beta_0$ and $\langle(k,l),w_2\rangle = \beta_4 - \beta_0$ does not change the number of non-degenerate positive solutions of~\eqref{eq:s_5}. This gives a system of the form~\eqref{eq:s_1} with the same number of non-degenerate positive solutions as~\eqref{eq:s_2}.
\end{proof}
 Consider a system~\eqref{eq:s_1} satisfying all the hypotheses of Lemma~\ref{L:Normal}. Since we are interested in its non-degenerate positive solutions, we may assume that $w_0=(0,0)$. Moreover, without loss of generality, we may assume that $a_1=b_1=1$. For the simplicity of further computations, we make the following change of coordinates. Let $m_1$ be the greatest common divisor of the coordinates of $w_1$. Setting $y_1=z^{\frac{w_1}{m_1}}$ and choosing any basis of $\mathbb{Z}^2$ with first vector $\frac{1}{m_1}\cdot w_1$, we get a monomial change of coordinates $(z_1,z_2)\mapsto (y_1,y_2)$ of $(\mathbb{RK}^*)^2$ such that $z^{w_1}=y_1^{m_1}$ and $z^{w_2}=y_1^{m_2}y_2^{n_2}$. Replacing $y_2$ by $y_2^{-1}$ if necessary, we assume that $n_2>0$. Indeed, $n_2\neq 0$, since by assumption the support of~\eqref{eq:s_1} is highly non-degenerate. With respect to these new coordinates, we obtain the system 
 
\begin{equation}\label{eq:mod:s_1}
  \begin{array}{ccccl}
    \displaystyle a_0 + y_1^{m_1} + a_2y_1^{m_2}y_2^{n_2} + a_3 t^{\alpha}y_1^{m_3}y_2^{n_3} & = & 0,\\ [4 pt]
    
    \displaystyle b_0 + y_1^{m_1} + b_2y_1^{m_2}y_2^{n_2} + b_4t^{\beta}y_1^{m_4}y_2^{n_4} & = & 0,\\
  \end{array}
\end{equation} that has the same number of non-degenerate positive solutions as~\eqref{eq:mod:s_1}. In what follows, we will work on a \textbf{normalized system} of the form~\eqref{eq:mod:s_1}, i.e. a highly non-degenerate system~\eqref{eq:mod:s_1} that satisfies the hypothesis of Lemma~\ref{L:Normal}. We will state two results, the proof of which are contained in~\cite{E16}, that are important for the construction.

A normal fan of a 2-dimensional convex polytope in $\mathbb{R}^2$ is the complete fan with apex at the origin, and 1-dimensional cones directed by the outward normal vectors of the $1$-faces of this polytope. Recall that $(0,0)$, $(m_1,0)$ and $(m_2,n_2)$ do not belong to a line (since~\eqref{eq:mod:s_1} is highly non-degenerate) and denote by $\Delta$ the triangle with vertices $(0,0)$, $(m_1,0)$ and $(m_2,n_2)$. Let $\mathcal{E}\subset\mathbb{R}^2$ denote the normal fan of $\Delta$. The fan $\mathcal{E}$ together with $\Delta$ are represented in Figure~\ref{Fig:BaseFans:EandGen}. The 1-dimensional cones of $\mathcal{E}$ are $\mathsf{L}_0=\{\lambda(0,-m_1)|\ \lambda\geq0\}$, $\mathsf{L}_1=\{\lambda(n_2, m_1 - m_2)|\ \lambda\geq0\}$ and $\mathsf{L}_2=\{\lambda(-n_2, m_2)|\ \lambda\geq0\}$. Let $\mathsf{C}_0$ (resp. $\mathsf{C}_1$, $\mathsf{C}_2$) denote the $2$-dimensional cone generated by the two vectors $(0,-m_1)$ and $(-n_2, m_2)$ (resp. $(0,-m_1)$ and $(n_2, m_1 - m_2)$, $(n_2, m_1 - m_2)$ and $(-n_2, m_2)$), see Figure~\ref{Fig:BaseFans:EandGen}. In what follows, for $i=0,1,2$, let $\accentset{\circ}{\mathsf{C}}_i$ denote the relative interior of $\mathsf{C}_i$ and $\accentset{\circ}{\mathsf{L}}_i$ denote the relative interior of $\mathsf{L}_i$. Finally, denote by $T_1$ (resp. $T_2$) the tropical curve associated to the first (resp. second) equation of~\eqref{eq:mod:s_1}.\\

\begin{theorem}[\cite{E16}] \label{Th.transv}
 For $i=0,1,2$, the relatively open $2$-cone $\accentset{\circ}{\mathsf{C}}_i$ cannot contain more than one tropical transversal intersection point of~\eqref{eq:mod:s_1}. Moreover, a $1$-cone of $\mathcal{E}$ does not contain a transversal intersection point of $T_1$ and $T_2$. Finally, if $T_1$ and $T_2$ intersect non-transversally at a cell $\xi$, then $\xi$ is contained in a 1-cone of the fan $\mathcal{E}$.
 \end{theorem}

\begin{proposition}[\cite{E16}]\label{Prop:TransvToTransv}
Assume that $T_1$ and $T_2$ intersect transversally at a point $v\in\accentset{\circ}{\mathsf{C}}_i$ for some $i\in\{0,1,2\}$. Then $\coef(a_i)\coef(a_3)<0$, $\coef(b_i)\coef(b_4)<0$ iff $v$ is the valuation of a positive solution of~\eqref{eq:mod:s_1}.
\end{proposition}

\begin{figure}[H]
\centering
\includegraphics[scale=0.8]{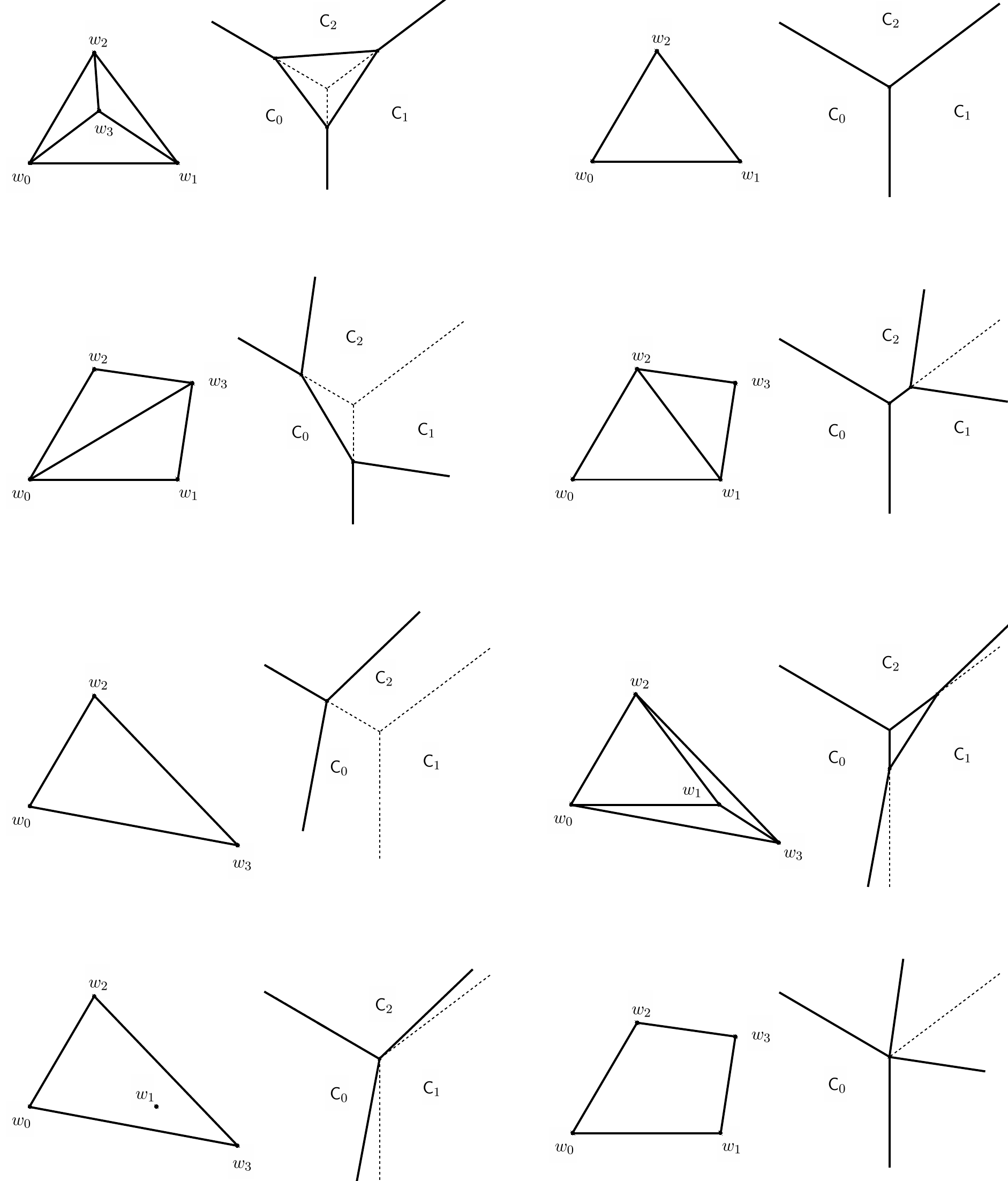}
\caption{Disposition of $T_1$ with respect to the fan $\mathcal{E}$ (together with its dual subdivision). These are all the possible configurations of $T_1$ (up to transformation) with respect to the fan $\mathcal{E}$. Since $T_2$ satisfies similar configurations, Figure~\ref{fig4} gives an idea of why Theorem~\ref{Th.transv} holds true.}
\label{fig4}
\end{figure}

\begin{figure}[H]
\centering
\includegraphics[scale=0.888]{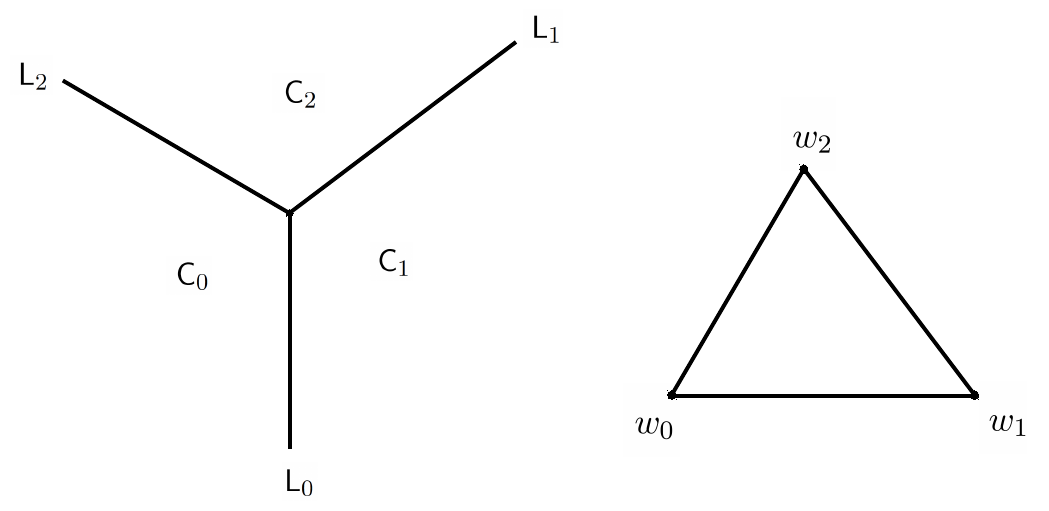}
\caption{The fan $\mathcal{E}$ together with its dual triangle.}
\label{Fig:BaseFans:EandGen}
\end{figure}

\subsection{Construction} In what follows, we construct a system~\eqref{eq:mod:s_1} having seven positive solutions. Theorem 1.16 of~\cite{E16} implies that if $\alpha\neq\beta$ or $\alpha=\beta<0$, then~\eqref{eq:mod:s_1} has at most six positive solutions.  Therefore, assume henceforth that $\alpha=\beta>0$. It is easy to deduce from equations appearing in~\eqref{eq:mod:s_1} that, since $\alpha,\beta\geq 0$, the tropical curves $T_1$ and $T_2$ intersect non-transversally at a point $v_0$ of type (III) that is the origin of $\mathcal{E}$. In order to study the positive solutions of~\eqref{eq:mod:s_1} with valuation $v_0$, we first consider the system \begin{equation}\label{E_2-E_1}
  \begin{array}{lcccl}
    \displaystyle  a_0 + y^{m_1}_1 + a_2y^{m_2}_1y^{n_2}_2 + a_3 t^{\alpha} y^{m_3}_1y^{n_3}_2 & = & 0,
    \\[10pt]
    \displaystyle  c_0t^{\gamma_0} + c_2t^{\gamma_2}y^{m_2}_1y^{n_2}_2 - a_3 t^{\alpha} y^{m_3}_1y^{n_3}_2 + b_4t^{\beta}y^{m_4}_1y^{n_4}_2 & = & 0,\\
  \end{array}
 \end{equation} with $c_it^{\gamma_i} = b_i - a_i$, $\ord(c_i)=0$ and $\gamma_i\geq 0$ for $i=0,2$. Since the second equation of~\eqref{E_2-E_1} is obtained by substracting the first equation of~\eqref{eq:mod:s_1} from its second one, this system has the same number of non-degenerate positive solutions as~\eqref{eq:mod:s_1}. The case-by-case study done in~\cite{E16} shows that we can hope to obtain a system~\eqref{eq:mod:s_1} having seven positive solutions if we have \begin{equation}\label{eq:inequalities1}
\coef(a_i)=\coef(b_i)\text{\quad for\quad}  i=0,2,\text{\quad and\quad} \alpha=\beta=\gamma_2<\gamma_0.
\end{equation} One possible disposition of the seven solutions is the following (see Figure~\ref{SevenPositiveS1}).\\

\begin{itemize}

 \item The common vertex $v_0$ is the valuation of five positive solutions,\\
 
 \item the $2$-cone $\mathsf{C}_2$ of $\mathcal{E}$ contains a transversal intersection $p$, and\\
 
 \item the $1$-cone $\mathsf{L}_0$ of $\mathcal{E}$ contains the valuation $q$ of one positive solution.\\
 
 \end{itemize}

\subsubsection{Reduced system at $v_0$.} Note that from~\eqref{eq:inequalities1} we deduce that the reduced system of~\eqref{E_2-E_1} with respect to $v_0$ is \begin{equation}\label{eq:red:5(0)}
   \begin{array}{lllllll}
    \displaystyle  \coef(a_0) & + &\displaystyle y^{m_1}_1 & + & \displaystyle \coef(a_2)y^{m_2}_1y^{n_2}_2 & = & 0,
    \\[10pt]
 \displaystyle \coef(b_4)y_1^{m_4}y_2^{n_4} & - & \displaystyle \coef(a_3)y_1^{m_3}y_2^{n_3}  & +&\displaystyle  \coef(c_2)y_1^{m_2}y_2^{n_2} & = & 0.
  \end{array}
   \end{equation} Such a system has at most five positive solutions. Indeed, since this is a system of two trinomials in two variables (see~\cite{LRW03}). Without loss of generality, we may assume that $\coef(a_0),\coef(a_3)<0$, and doing a suitable monomial change of coordinates followed by a multiplication of each equation of~\eqref{eq:red:5(0)} by a constant, we assume in addition that $\coef(a_0)=\coef(a_3)=-\coef(a_2)=-1$. Therefore, the reduced system of~\eqref{E_2-E_1} with respect to $\{v_0\}$ is now\begin{equation}\label{eq:red:5}
 \begin{array}{lllllll}
    \displaystyle  -1 & + &\displaystyle y^{m_1}_1 & + & \displaystyle y^{m_2}_1y^{n_2}_2 & = & 0,
    \\[10pt]
 \displaystyle \coef(b_4)y_1^{m_4}y_2^{n_4} & - & \displaystyle y_1^{m_3}y_2^{n_3}  & +&\displaystyle  \coef(c_2)y_1^{m_2}y_2^{n_2} & = & 0.
  \end{array}
\end{equation} Assume that the open $1$-cone $\oc{\mathsf{L}}_0$ of $\mathcal{E}$ contains the valuation of one (which is the maximum possible for this case) positive solution of~\eqref{eq:mod:s_1}. Then both $n_3$ and $n_4$ are positive.  Therefore, since $\alpha>0$, both $T_1$ and $T_2$ do not have a vertex in $\mathsf{L}_0$ (see Figure~\ref{SevenPositiveS1} for example). Assume furthermore that $T_1$ and $T_2$ do not intersect non-transversally at a point of type (III) belonging to the relative interior of a $1$-cone of $\mathcal{E}$.

We start our construction by finding a system~\eqref{eq:red:5} that has five positive solutions. Since systems of two trinomials in two variables having five positive solutions are hard to generate (c.f.~\cite{DRRS07}), we will borrow one from the literature and base our construction upon it. 

First, we define a univariate function $f$ such that for some constant $c$, the equation $f=c$ has the same number of solutions in $]0,1[$ as that of positive solutions of~\eqref{eq:red:5}. We write the first equation of~\eqref{eq:red:5} as $y_2 = x^{k}(1 - x)^{l}$, where $x:=y_1^{m_1}$, $\displaystyle k=-m_2/(m_1 n_2)$ and $l =1/n_2$. It is clear that $y_1,y_2>0\Leftrightarrow x\in I_0:=]0,1[$. Since we are looking for solutions of~\eqref{eq:red:5} with non-zero coordinates, we divide its second equation by $y_1^{m_2}y_2^{n_2}$. Plugging $y_1$ and $y_2$ in the second equation of~\ref{eq:red:5}, we get

 \begin{align} \label{eq:fto5}
   \coef(c_2)  + x^{k_3}(1-x)^{l_3} + \coef(b_4)x^{k_4}(1-x)^{l_4}   =  0,
 \end{align} where $\displaystyle k_i=\frac{m_in_2 - m_2n_i}{m_1n_2}$ and $\displaystyle l_i=\frac{n_i-n_2}{n_2}$ for $i=3,4$. The number of positive solutions of~\eqref{eq:red:5} is equal to the number of solutions of~\eqref{eq:fto5} in $I_0$. Therefore we want to compute values of $\coef(c_2)$, $\coef(b_4)$ and $(m_i,n_i)$ for $i=1,2,3,4$ such that $f(x)=-\coef(c_2)$ has five solutions in $I_0$, where  
\begin{align}
   f(x) :=  x^{k_3}(1-x)^{\l_3} + \coef(b_4)\cdot x^{k_2}(1-x)^{l_2}.
\end{align}

 Note that the function $f$ has no poles in $I_0$, thus by Rolle's theorem we have $\sharp\{x\in I_0\ | f(x)= 1\}\ \leq\sharp\{x\in I_0\ |f'(x)=0\} + 1$. The derivative $f'$ is expressed as $$\displaystyle x^{k_3 -1}(1-x)^{l_3-1}\rho_3(x)  + a_4x^{k_4 -1}(1-x)^{l_4 - 1}\rho_4(x),$$ where $\rho_i(x)=k_i - (k_i +l_i)x$ for $i=3,4$. For $x\in~]0,1[$, we have $\displaystyle f'(x)=0\Leftrightarrow \phi(x)=1$, where \begin{equation}\label{eq:rational}
\phi(x):=- \coef(b_4)\frac{x^{k_4 -k_3}(1-x)^{l_4-l_3}\rho_4(x)}{\rho_3(x)}.
\end{equation} Consider the system 
\begin{align} \label{ex:s_1}
 x^6 + (44/31)y^3 - y  = y^6 + (44/31)x^3 - x  =  0,
\end{align} taken from~\cite{DRRS07}, which has five positive solutions. The rational function ~\eqref{eq:rational}, associated to~\eqref{ex:s_1} is $$\phi_0(x)=(44/31)^{5/6} \cdot \frac{x^{1/6}(1-x)^{1/3}(-11/4 + 9x/4)}{(-35/12 + 11x/4)}.$$ Thus, if

\begin{equation}\label{eq:equalities}
\begin{array}{lllllll}
    \displaystyle  \coef(b_4)=-\left(\frac{44}{31}\right)^{\frac{5}{6}}, & \displaystyle k_4 - k_3 = \frac{1}{6}, &  \displaystyle l_4 - l_3 =\frac{1}{3}, \\[10pt]
 \displaystyle k_4 = -\frac{11}{4} & \text{and} & \displaystyle  k_3 = -\frac{35}{12},
  \end{array}
\end{equation} then $\phi(x)=1$ has four positive solutions in $I_0$. Assume that equalities in~\eqref{eq:equalities} hold true. Plotting the function $f:\R\rightarrow\R$, $x\mapsto f(x)$, we get that the graph of $f$ has four critical points contained in $I_0$ with critical values situated below the $x$-axis. Moreover, this graph intersects transversally the line $\{y=-0.36008\}$ in five points with the first coordinates belonging to $I_0$. Therefore, the equation $f(x)=-0.36008$ has five non-degenerate positive solutions in $I_0$. 

\subsubsection{Choosing the monomials.} In what follows, we find $(m_i,n_i)\in\mathbb{Z}^2$ for $i=1,2,3,4$, satisfying the equalities in~\eqref{eq:equalities} so that~\eqref{eq:red:5} has five non-degenerate positive solutions. Recall that $m_1,n_2>0$ (since~\eqref{eq:mod:s_1} is normalized) and assume that $m_2$ is also positive. The equalities in~\eqref{eq:equalities} show that $l_i>0$, $k_i<0$ and $k_i<l_i$ for $i=3,4$, therefore we have $0<n_2<n_i$, $m_in_2 - n_im_2<0$ and $ (m_i - m_1)n_2 - n_i(m_2-m_1)<0$ for $i=3,4$. Plotting the three points $(0,0)$, $(m_1,0)$ and $(m_2,n_2)$, we deduce from the latter inequalities that the points $(m_3,n_3)$ and $(m_4,n_4)$ belong to the region $B_1$ of Figure~\ref{Fig:RegionB1B11}.
\begin{figure}[H]
\centering
\includegraphics[scale=0.9]{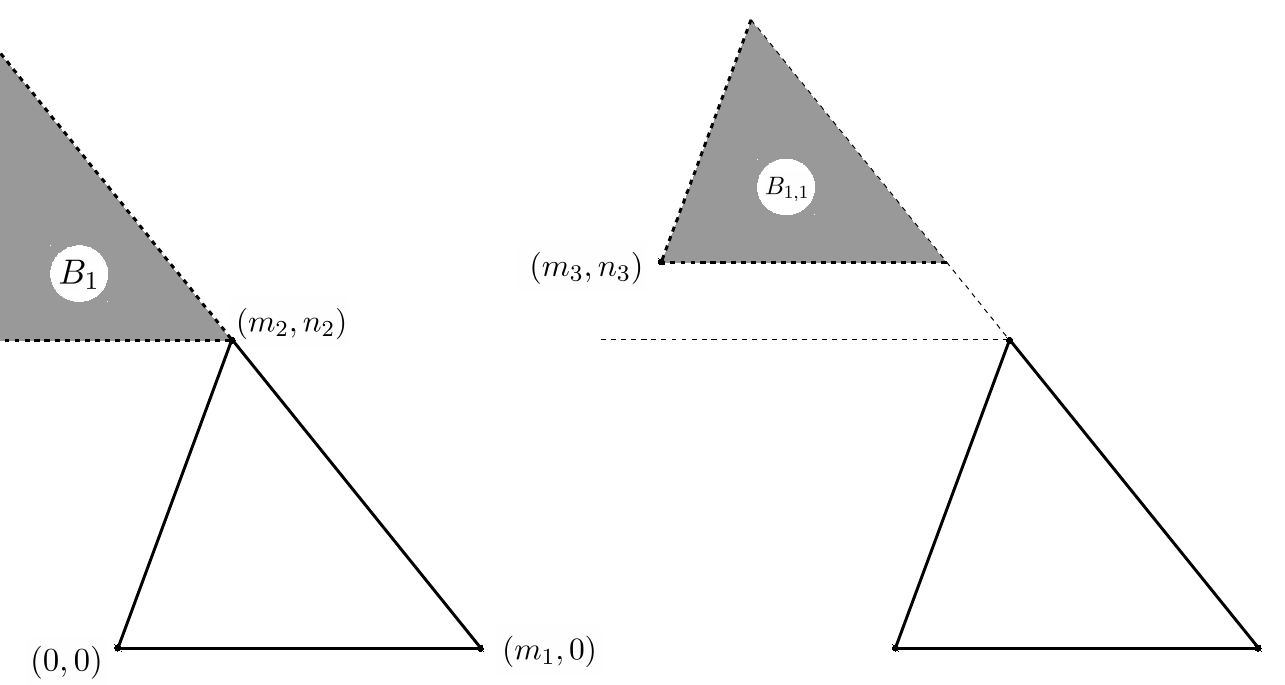}
\caption{The region $B_1$ and triangle $B_{1,1}$}
\label{Fig:RegionB1B11}
\end{figure}

We also deduce from equalities in~\eqref{eq:equalities} that $l_4>l_3$ and $k_4>k_3$, and thus $n_4>n_3$ and $(m_4-m_3)n_2 - (n_4-n_3)m_2>0$. Fixing $(m_3,n_3)$ in the region $B_1$, we obtain that $(m_4,n_4)$ belongs to the triangle $B_{1,1}$ depicted in Figure~\ref{Fig:RegionB1B11}.

 Note that the vertex $v_1\in\mathsf{L}_2$ (resp. $v_2\in\mathsf{L}_2$) of $T_1$ (resp. $T_2$) has coordinates $$\frac{\alpha}{m_3n_2-n_3m_2}(n_2,-m_2)\quad \left(\text{resp.}\quad\frac{\alpha}{m_4n_2-n_4m_2}(n_2,-m_2) \right),$$ and thus from $m_3n_2 - n_3m_2<m_4n_2 - n_4m_2<0$, we deduce that the first coordinate of $v_2$ is smaller than that of $v_1$ (see Figure~\ref{SevenPositiveS1}).
  
All these restrictions impose that there exists a transversal intersection point of $T_1$ and $T_2$ in $\mathsf{C}_2$ (see Figure~\ref{SevenPositiveS1} for example). Moreover, since $\coef(b_4)<0$ (see~\eqref{eq:equalities}), $\coef(a_3)=-1$ (from~\eqref{eq:red:5}) and $\coef(a_0)=\coef(b_0)=-1$, Proposition~\ref{Prop:TransvToTransv} shows that the intersection point $p$ is the valuation of a positive solution of~\eqref{eq:mod:s_1}. The constant $\coef(c_0)$ should be a negative number so that~\eqref{eq:mod:s_1} has a positive solution with valuation in $\mathsf{L}_0$. This constant can take any negative value, and for computational reasons we choose it to be $-0.36008$.
%
  
According to this analysis, a valid choice of exponents and coefficients of~\eqref{eq:mod:s_1} is $m_1=6$, $(m_2,n_2)=(3,6)$, $(m_3,n_3)=(-14,7)$, $(m_4,n_4)=(-12,9)$, $a_0=-1$, $a_2=1$, $a_3=-t^{\alpha}$, $b_0=-1  + 0.36008t^{\gamma_0}$ (verifying $\gamma_0>\alpha$), $b_2=-1+t^{\alpha}$ and $b_4= -\left(44/31\right)^{5/6}t^{\alpha}$. Therefore, the system \begin{equation}\label{part.form.const.7}
  \begin{array}{ccccll}
    \displaystyle  -1 & +~ \displaystyle y_1^{6} & +~\displaystyle y_1^3y_2^6  & -~ \displaystyle t^{\alpha}y_1^{-14}y_2^{7} & =& 0, \\ [6 pt]
    
    \displaystyle  -1  + 0.36008t^{\gamma_0} & +~\displaystyle y_1^{6} & +~\displaystyle (1-0.36008t^{\alpha})y_1^3y_2^6 & -~\displaystyle (44/31)^{\frac{5}{6}}t^{\alpha}y_1^{-12}y_2^9 & =& 0, \\    
  \end{array}
\end{equation} which has seven non-degenerate positive solutions, proves Theorem~\ref{Th:Main}.

\begin{figure}[H]
\centering
\includegraphics[scale=1.49]{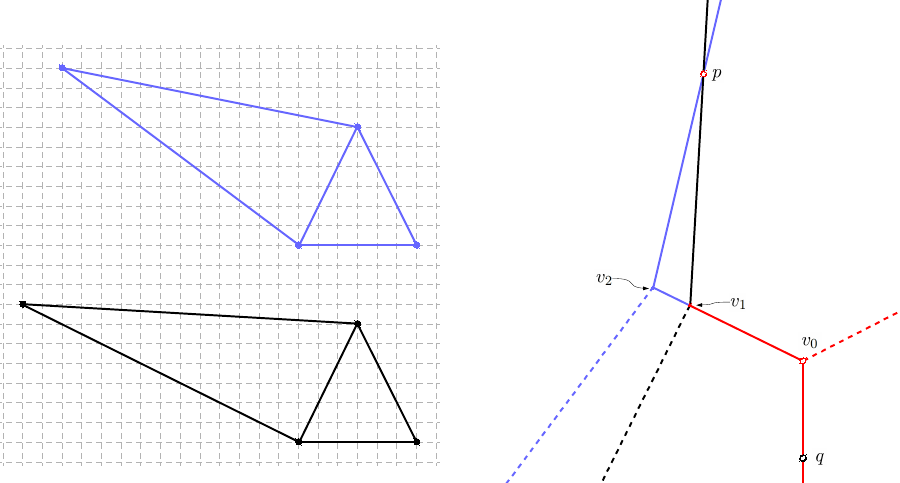}
\caption{Newton polytopes and tropical curves associated to a normalized system having seven positive solutions.}\label{SevenPositiveS1}
\end{figure} 

\subsubsection{A software computation} Using Maple 17 as well as the libraries FGb and RS, Pierre-Jean Spaenlehauer~\cite{Spa16} provided us with a computation he made of the non-degenerate positive solutions of a system~\eqref{part.form.const.7} for $\gamma_0=7$ and $\alpha=1$ that goes as follows. For computational reasons, he has replaced the real number $(44/31)^{5/6}$ in~\eqref{part.form.const.7} by the fraction 

$$\frac{26807502408507435267952730104920543812845885439976}{20022295568917288472920446333489413342983920443429}$$ which approximates $(44/31)^{5/6}$. For $t=1/100~000$, the computer software has found seven positive solutions. An approximation of these solutions goes as follows. 

$$(0.99999,~0.00001),~(0.99171,~0.60681),~(0.96651,~0.76771),~(0.95765,~0.79907),$$

$$(0.95201,~0.81642),~(0.88602,~ 0.95151),~(0.53645,~1.61099).$$

\bibliographystyle{alpha}					   

\bibliography{Biblio}          

\def\cprime{$'$}
\begin{thebibliography}{BLdM12}

\bibitem[BB13]{BB13}
Beno{\^{\i}}t Bertrand and Fr{\'e}d{\'e}ric Bihan.
\newblock Intersection multiplicity numbers between tropical hypersurfaces.
\newblock In {\em Algebraic and combinatorial aspects of tropical geometry},
  volume 589 of {\em Contemp. Math.}, pages 1--19. Amer. Math. Soc.,
  Providence, RI, 2013.

\bibitem[Ber75]{Ber75}
D.~N. Bernstein.
\newblock The number of roots of a system of equations.
\newblock {\em Funkcional. Anal. i Prilo\v zen.}, 9(3):1--4, 1975.

\bibitem[Bih07]{B07}
Fr{\'e}d{\'e}ric Bihan.
\newblock Polynomial systems supported on circuits and dessins d'enfants.
\newblock {\em J. Lond. Math. Soc. (2)}, 75(1):116--132, 2007.

\bibitem[Bih14]{B14}
Fr{\'e}d{\'e}ric Bihan.
\newblock Irrational mixed decomposition and sharp fewnomial bounds for
  tropical polynomial systems.
\newblock {\em arXiv preprint arXiv:1410.7905~(To appear in Discrete and
  Computational Geometry)}, 2014.

\bibitem[BLdM12]{Br-deMe11}
Erwan~A. Brugall{\'e} and Lucia~M. L{\'o}pez~de Medrano.
\newblock Inflection points of real and tropical plane curves.
\newblock {\em J. Singul.}, 4:74--103, 2012.

\bibitem[BR90]{BR90}
O.~Bottema and B.~Roth.
\newblock {\em Theoretical kinematics}.
\newblock Dover Publications, Inc., New York, 1990.
\newblock Corrected reprint of the 1979 edition.

\bibitem[BS07]{BS07}
Fr{\'e}d{\'e}ric Bihan and Frank Sottile.
\newblock New fewnomial upper bounds from {G}ale dual polynomial systems.
\newblock {\em Mosc. Math. J.}, 7(3):387--407, 573, 2007.

\bibitem[Byr89]{By89}
C.~I. Byrnes.
\newblock Pole assignment by output feedback.
\newblock In {\em Three decades of mathematical system theory}, volume 135 of
  {\em Lecture Notes in Control and Inform. Sci.}, pages 31--78. Springer,
  Berlin, 1989.

\bibitem[DRRS07]{DRRS07}
Alicia Dickenstein, Jean-Maurice Rojas, Korben Rusek, and Justin Shih.
\newblock Extremal real algebraic geometry and {$\mathcal A$}-discriminants.
\newblock {\em Mosc. Math. J.}, 7(3):425--452, 574, 2007.

\bibitem[EH16]{E16}
Boulos El~Hilany.
\newblock {\em Tropical geometry and polynomial systems}.
\newblock PhD thesis, Comunaut\'e Universit\'e Grenoble Alpes, 2016.
\newblock \url{https://www.math.uni-tuebingen.de/user/boel/Thesis.pdf}.

\bibitem[GH02]{GH02}
Karin Gatermann and Birkett Huber.
\newblock A family of sparse polynomial systems arising in chemical reaction
  systems.
\newblock {\em J. Symbolic Comput.}, 33(3):275--305, 2002.

\bibitem[Haa02]{H02}
Bertrand Haas.
\newblock A simple counterexample to {K}ouchnirenko's conjecture.
\newblock {\em Beitr\"age Algebra Geom.}, 43(1):1--8, 2002.

\bibitem[IMS09]{IMS09}
Ilia Itenberg, Grigory Mikhalkin, and Eugenii~I Shustin.
\newblock {\em Tropical algebraic geometry}, volume~35.
\newblock Springer Science \& Business Media, 2009.

\bibitem[Kap00]{Kap00}
Mikhail~M Kapranov.
\newblock Amoebas over non-archimedean fields.
\newblock {\em preprint}, 2000.

\bibitem[Kat09]{Kat09}
Eric Katz.
\newblock A tropical toolkit.
\newblock {\em Expositiones Mathematicae}, 27(1):1--36, 2009.

\bibitem[Kho91]{Kh91}
A.~G. Khovanski{\u\i}.
\newblock {\em Fewnomials}, volume~88 of {\em Translations of Mathematical
  Monographs}.
\newblock American Mathematical Society, Providence, RI, 1991.
\newblock Translated from the Russian by Smilka Zdravkovska.

\bibitem[Kus75]{Kus75}
Anatoli Kushnirenko.
\newblock A newton polyhedron and the number of solutions of a system of k
  equations in k unknowns.
\newblock {\em Uspekhi Mat. Nauk}, 30:261--269, 1975.

\bibitem[LRW03]{LRW03}
Tien-Yien Li, Jean-Maurice Rojas, and Xiaoshen Wang.
\newblock Counting real connected components of trinomial curve intersections
  and {$m$}-nomial hypersurfaces.
\newblock {\em Discrete Comput. Geom.}, 30(3):379--414, 2003.

\bibitem[Mik04]{Mik04}
Grigory Mikhalkin.
\newblock Decomposition into pairs-of-pants for complex algebraic
  hypersurfaces.
\newblock {\em Topology}, 43(5):1035--1065, 2004.

\bibitem[Mik06]{Mik06}
Grigory Mikhalkin.
\newblock Tropical geometry and its applications.
\newblock In {\em International {C}ongress of {M}athematicians. {V}ol. {II}},
  pages 827--852. Eur. Math. Soc., Z\"urich, 2006.

\bibitem[MR05]{MR05}
Grigory Mikhalkin and Johannes Rau.
\newblock Tropical geometry.
\newblock {\em Book in preparation}, 1(38):343, 2005.

\bibitem[MS15]{MS15}
Diane Maclagan and Bernd Sturmfels.
\newblock {\em Introduction to tropical geometry}, volume 161.
\newblock American Mathematical Soc., 2015.

\bibitem[OP13]{OP13}
Brian Osserman and Sam Payne.
\newblock Lifting tropical intersections.
\newblock {\em Documenta Mathematica}, 18:121--175, 2013.

\bibitem[Rab12]{Rab12}
Joseph Rabinoff.
\newblock Tropical analytic geometry, newton polygons, and tropical
  intersections.
\newblock {\em Advances in Mathematics}, 229(6):3192--3255, 2012.

\bibitem[Ren15]{Ren15}
Arthur Renaudineau.
\newblock {\em Constructions de surfaces alg{\'e}briques r{\'e}elles}.
\newblock PhD thesis, Paris 6, 2015.

\bibitem[Rul01]{Rul01}
Hans Rullgard.
\newblock Polynomial amoebas and convexity.
\newblock {\em Preprint, Stockholm University}, 2001.

\bibitem[Spa]{Spa16}
Pierre-Jean Spaenlehauer.
\newblock Personal communication.

\bibitem[Stu94]{S94}
Bernd Sturmfels.
\newblock Viro's theorem for complete intersections.
\newblock {\em Ann. Scuola Norm. Sup. Pisa Cl. Sci. (4)}, 21(3):377--386, 1994.

\end{thebibliography}

\end{document}